\documentclass[11pt]{elsarticle}
\usepackage[a4paper,margin=3cm]{geometry}
\usepackage{times}
\usepackage{amsthm}
\usepackage{amsmath}
\usepackage{amssymb}
\usepackage{mathtools}
\usepackage{relsize}%
\usepackage[utf8]{inputenc}
\usepackage{enumerate}
\usepackage[pdfkeywords={science, physics, math},pagebackref]{hyperref}
\usepackage{float}
\usepackage{array}
\usepackage{tikz}
\usetikzlibrary{calc}
\usetikzlibrary{decorations.pathmorphing,patterns}
\usetikzlibrary{matrix}
\usepackage{cancel}

\definecolor{darkgreen}{rgb}{0.0,0.35,0.0}

\usepackage{todonotes}

\makeatletter
\def\ps@pprintTitle{%
  \let\@oddhead\@empty
  \let\@evenhead\@empty
  \def\@oddfoot{\reset@font\hfil\thepage\hfil}
  \let\@evenfoot\@oddfoot
}
\makeatother

\usepackage[framemethod=tikz]{mdframed}
\theoremstyle{plain}
\newtheorem{theorem}{Theorem}[section]
\newtheorem*{theorem*}{Theorem}
\newtheorem{lemma}[theorem]{Lemma}

\newtheorem*{corollary*}{Corollary}
\newtheorem{claim}[theorem]{Claim}
\newtheorem{proposition}[theorem]{Proposition}

\theoremstyle{definition}

\newtheorem{remark}[theorem]{Remark\textbf{}}
\newtheorem*{remark*}{Remark}
\newtheorem{definition}[theorem]{Definition}

\newtheorem{example}[theorem]{Example}

\def\G{\mathcal{G}}

\def\mD{\mathbb{D}}

\def\mC{\mathbb{C}}
\def\C{\mathcal{C}}

\def\F{\mathcal{F}}

\def\aE{\text{\AE}}
\def\bd{\mathbf{bd}}
\def\deg{\mathrm{deg}}

\def\fB{\mathfrak{B}}
\def\fD{\mathfrak{D}}

\def\lbb{\operatorname{LEG}}
\def\lbd{\operatorname{LEG}}

\def\FG{fully graphic}

\def\Nec{\operatorname{\overline{\Gamma}}}

\newcommand\arXiv[1]{{\sl arXiv:\/}{\bf #1}}
\newcommand\DOI[1]{{DOI:#1}}
\newcommand\junk[1]{}

\newcommand{\pppp}[1]{\partial^{++}#1}

\newcommand{\appropto}{\mathrel{\vcenter{
  \offinterlineskip\halign{\hfil$##$\cr
    \propto\cr\noalign{\kern2pt}\sim\cr\noalign{\kern-2pt}}}}}

\newcommand{\aA}{\mathbf{A}}

\newcommand{\dD}{\mathbf{D}}

\newcommand{\oG}{\overrightarrow{G}}

\begin{document}
\begin{frontmatter}
\title{P-stable degree sequence regions for bipartite and directed graphs}
\author[renyi]{Péter L. Erdős}
\author[NEU]{Gábor Lippner\fnref{lg}}
\author[NEU]{Na'ama Nevo}
\author[renyi]{Lajos Soukup}
\address[renyi]{HUN-REN Alfréd Rényi Institute of Mathematics, Reáltanoda u 13--15 Budapest, 1053 Hungary\\
\texttt{$<$erdos.peter,soukup.lajos$>$@renyi.hun-ren.hu}}
\address[NEU]{Northeastern University, 360 Huntington Ave, Boston, MA, 02115\\ $<$g.lippner,nevo.n$>$@northeastern.edu}
\fntext[lg]{GL was supported in part by Simons Foundation Collaboration Grant No 953420}

\begin{abstract}
The switch Markov chain is a natural method for sampling graph realizations with a prescribed degree sequence, and $P$-stability is a general sufficient condition for its rapid mixing.  Since $P$-stability is often difficult to verify directly, we develop an approach based on the simpler and more readily verifiable property of \emph{full graphicality}.

We consider bidegree sequence regions obtained by fixing the two class sizes, the common degree sum, and lower and upper bounds for the degrees in each class.  Such a region is fully graphic if every bidegree sequence satisfying these constraints has a bipartite realization.  Our main result shows that every fully graphic region of this form is $P$-stable.  More precisely, for class sizes $n$ and $m$, the perturbation ratio is at most $(nm)^5$.  It follows from a known theory that the switch Markov chain is rapidly mixing on every such region.  We also show that full graphicality can be tested on a canonical extremal bidegree sequence and use this characterization to derive several
new explicit sufficient conditions.

By encoding loopless directed graphs as bipartite graphs, we develop a parallel theory for directed degree sequences.  This gives analogous extremal characterizations, sufficient conditions for full graphicality, and $P$-stability and rapid-mixing results for directed degree-sequence regions. The encoding process allows most corresponding statements for the two cases to be formulated and proved in one run.
\end{abstract}


\end{frontmatter}

\section{Introduction}\label{sec:intro}
Generating a graph uniformly at random from all labeled realizations of
a prescribed degree sequence is a basic problem in graph sampling.  A
natural approach is the switch Markov chain, whose transitions replace
two edges by two non-edges while preserving every vertex degree
\cite{RJB}.  Kannan, Tetali, and Vempala conjectured that this chain is
rapidly mixing for every graphic degree sequence \cite{KTV97}. While this manuscript has been prepared, a new preprint of Fu, Qin, and Wang  \cite{Fu} establishes polynomial mixing for arbitrary bipartite degree sequences.

One general route to rapid mixing is provided by $P$-stability, a
notion introduced by Jerrum and Sinclair \cite{JS90}.  Informally,
$P$-stability requires the number of realizations of suitably perturbed
degree sequences to be at most a polynomial factor larger than the
number of realizations of the original sequence.  The unified approach
of \cite{P-stable} shows that $P$-stability implies rapid mixing of the
switch Markov chain and encompasses the previously known rapidly
mixing families.  The difficulty is that $P$-stability is a counting
condition and is usually hard to verify directly.

A different approach to prove such stability property was developed in \cite{fully},
through the notion of a \emph{fully graphic} degree-sequence region.  Such
a region is specified by elementary constraints on the degrees and is fully graphic when every sequence satisfying those constraints has a realization.  Thus
full graphicality is an existence property rather than a counting property.
For simple graphs, \cite{fully} used this property to obtain new $P$-stable
families, and \cite{Simple-fully} developed a general connection between full graphicality and $P$-stability.  The purpose of the present paper is to establish the corresponding theory for bipartite and directed graphs.

\bigskip\noindent Before we turn our attention to this different approach we describe in short the problem of anonymizing in social networks, and one suggested approach to it. \textsl{Anonymizing} is the process of removing or modifying personally identifiable information (PII) from a dataset or file so that the original individual can no longer be identified. In the papers of Solas and Torra (\cite{solas1,solas2}) using $P$-stable approximation networks were suggested for that goal.

\bigskip\noindent
In the new approach, for bipartite graphs we consider regions obtained by fixing the two class sizes, a common degree sum, and lower and upper degree bounds in
each class. Then we describe a slight modification of a known approach, introduced by Gale (\cite{Gale}) in 1957, to study directed graphs via bipartite graphs with some auxiliary structures. We will call this construct as a \emph{partially paired bipartite graph} (or \textbf{PPBG} for short).
When the auxiliary structure is empty, then we retain the original bipartite graph notion. This extended notion is suitable to describe concepts and results in parallel for both graph classes.

Using this unifying approach we give a common characterization of full graphicality for both cases: it turns out that for proving the full graphicality of a region it is enough to test the graphicality of a single canonical degree sequence pair (denoted by LEG) for all feasible parameter sets.

The paper's main result then shows that every fully graphic partially paired bipartite region is $P$-stable.  More precisely, if the two class sizes are $n$ and $m$, then the relevant perturbation ratio is at most $(nm)^5$. The proof converts a realization of a perturbed degree sequence into a realization of the original sequence by reversing a short alternating path.  The essential structural step shows that if no such short path exists, then one can construct a non-graphic sequence inside the same region, contradicting full graphicality.

Finally we will use the LEG-characterization to derive several explicit fully graphic partially paired bipartite regions, which in turn provide substantial extensions of previously known $P$-stable regions.

The paper is organized as follows.  Section~\ref{sec:def} develops a suitable notation system to describe the bipartite and directed degree sequence problems in parallel. Section~\ref{sec:char} will develop a method to prove full graphicality for partially paired bipartite regions. Section~\ref{sec:forbid} will prove all fully graphic partially paired bipartite regions are $P$-stable as well. Finally Section~\ref{sec:examples} will provide several fully graphic bipartite and directed regions, extending several known regions as well as find new ones.

\section{Partially paired bipartite graphs}\label{sec:def}
\noindent In the first subsection we will recall some well established notations and facts about bipartite graphs, furthermore we will introduce the notion of degree sequence regions. In the next subsection we will discuss the fundamental facts about directed graphs and we will discuss our bipartite representation of directed graphs. Finally, in the very short third subsection we will discuss the notion of P-stability for PPBGs.

\subsection{Bipartite degree sequences}\label{sec:bipart}
\noindent An $(n,m)$-\textbf{bisequence} is an ordered pair of positive integer sequences
\[
(C,D)= ((c_1,\cdots, c_n), (d_1, \cdots, d_m) ).
\]
It is a $(n,m)$-\textbf{bidegree sequence} if
$$
1\le c_i\le m \quad (i\in [n]), \qquad 1\le d_j\le n\quad (j\in[m]), \qquad \sum_{i=1}^n c_i=\sum_{j=1}^m d_j.
$$
The common sum will be denoted by $\Sigma$.

A bidegree sequence $(C,D)$ is \textbf{graphic} if there exists a bipartite graph $G$ with no parallel edges and with pairwise disjoint labeled vertex classes
$$
U=\{ u_1,\ldots, u_n\}, \qquad W=\{w_1,\ldots, w_m\};
$$
such that
$$
d_G(u_i)=c_i\quad (i\in [n]), \qquad d_G(w_j)=d_j\quad (j\in [m]).
$$
Such a graph is called a \textbf{labeled realization} of $(C,D).$ We write
$\G(C,D)$ for the set of all labeled realizations and put $\G(C,D)=\varnothing$ when $(C,D)$ is not a graphic bidegree sequence.  The coordinates of $C$ and $D$ correspond to the fixed vertex labels.  When a graphicality criterion requires non-increasing sequences, we may reorder the coordinates by relabeling the vertices;
this does not change the number of realizations.

A \textbf{bidegree sequence region} is a set of bidegree sequences. If a formula $\varphi=\varphi(C,D)$ specifies such a set, we call this a
\textbf{bidegree region}  (or just region, for short) and denote it by
\begin{equation}\label{eq:region}
 \fB[\varphi] = \{ (C,D) : \varphi(C,D)\}.
\end{equation}
In particular, for positive integers $\Sigma; 0< \delta_2 < \delta_1 \le n;$  and $0< \gamma_2 < \gamma_1 \le m$  we consider the following region
\begin{multline}\label{eq:BSR}
\fB(\Sigma;n,\Gamma;m,\Delta)=
\biggl\{(C,D): |C|=n,\ |D|=m,\
\gamma_2\le c_i\le\gamma_1\ (i\in[n]),\\
\delta_2\le d_j\le\delta_1\ (j\in[m]),\
\sum_{i=1}^{n}c_i=\sum_{j=1}^{m}d_j=\Sigma \biggr\}.
\end{multline}
Here we use the $\Gamma=(\gamma_1,\gamma_2), \Delta=(\delta_1, \delta_2)$ shorthand. We will call this construction  a \textbf{basic region}. The region above can be nonempty only if
\begin{equation}\label{eq:nemkell}
 \max\{n\gamma_2,m\delta_2\}
 \le \Sigma\le
 \min\{n\gamma_1,m\delta_1\}.
\end{equation}
However, when one of the equalities holds, then we have a homogeneous sequence what we do not want to consider. Therefore in inequality \eqref{eq:nemkell} we do not allow equality.
A region is \textbf{simple} if it is a union of basic regions
$\fB(\Sigma;n,\Gamma;m,\Delta)$.  For fixed $n,m,\Gamma,\Delta$, we will consider the union of all basic regions $\fB(\Sigma;n,\Gamma;m,\Delta)$ as well, where $\Sigma$ runs all feasible values. Furthermore we call a region \textbf{very simple} if it is defined as
\begin{equation}\label{eq:BVSR}
 \fB(n,\Gamma;m,\Delta)
 =\bigcup_{\substack{\max\{n\gamma_2,m\delta_2\}<\Sigma\\
                     \Sigma<\min\{n\gamma_1,m\delta_1\}}}
   \fB(\Sigma;n,\Gamma;m,\Delta).
\end{equation}
Typically, a region $\fB$ may contain both graphic and non-graphic sequences. We will be specifically interested in regions in which every sequence is graphic. Such regions are called \textbf{fully graphic}. Thus, under this definition, the empty region is fully graphic vacuously.

\bigskip\noindent
Next we recall the  two most well known characterizations  of graphic bidegree sequences:
\begin{theorem}[Gale, Ryser - 1957, \cite{Gale,Ryser}]\label{th:ZZ1}
Let $(C,D)$ be an $(n,m)$-bidegree sequence and assume that $C=(c_1,\ldots,c_n)$ is non-increasing. It is graphic if and only if for any $k=1,2,\ldots,n$
\begin{equation}\label{eq:ZZ}
\sum_{i=1}^k c_i \le \sum_{j=1}^m \min (d_j, k).
\end{equation}
\end{theorem}

\noindent
In 1992 Zverovich and Zverovich developed the bipartite analog of the Tripathi-Vijay theorem (\cite{TV03}) for simple graphs:
\begin{theorem}[{\cite[Theorem 8]{ZZ}}]\label{th:ZZ2}
Under the assumptions above, it suffices to check the inequality at $k=n$ and at those indices $k<n$ for which
$c_k>c_{k+1}$.
\end{theorem}

\subsection{Directed degree sequences}\label{sec:direct}
\noindent A simple directed graph $\vec G$ on the underlying set $X$ of $n$ elements without \textbf{loop}s,  \textbf{parallel edge}s and \textbf{isolated} vertices has the \textbf{directed degree sequence}
$$
\scalebox{1.4}{(}(d^+(x_1),\cdots, d^+(x_n)) (d^-(x_1),\cdots,d^-(x_n)) \scalebox{1.4}{)}
$$
where vertices are listed in a non-increasing order in the out-degree sequence. A directed degree sequence $(C,D)$ is \textbf{graphic} if there exists a directed graph on the vertex set $X$ such that $\forall i=1,\cdots,n : d^+(x_i)=c_i \text{ and } d^-(x_i)=d_i.$ Since no loop is allowed, therefore each in- and out-degree is at most $n-1.$

\medskip
A suitable necessary and sufficient condition for a directed sequence being graphic was sought for a long time. The typical early solutions allowed loops and/or parallel edges. The most used necessary and sufficient condition  for graphicality of digraph sequences on $n$ vertices, without loops and/or parallel edges, is due to Berger \cite{Berger}
\begin{theorem}[Berger, Theorem 2]\label{th:berger}
The directed degree sequence with a non-increasing $C$ is \textbf{graphic} if and only if\ $\sum_{i=1}^n c_i=\sum_{i=1}^n d_i$ and we have
\begin{equation}\label{eq:Berger}
\sum_{i=1}^k\min (d_i,k-1) + \sum_{i=k+1}^n\min (d_i, k) \ge \sum_{i=1}^k c_i
\end{equation}
for each $k=1,\ldots,n-1$ where $c_k>c_{k+1}$  and for $k=n.$
\end{theorem}
\noindent This result is clearly the exact analog of the Tripathi-Vijay theorem.

\begin{remark}\label{rem:berger_order}~
The left hand side of \eqref{eq:Berger} is equal to  $\sum_{i=1}^n \min(d_i,k) - |\{i \leq k: d_i \geq k\}|$, from which it is clear that \eqref{eq:Berger} becomes tighter for each $k$ simultaneously when the $D$ sequence is rearranged to be monotone decreasing. Of course, this rearrangement makes the condition sufficient but not necessary anymore.
\end{remark}

\bigskip
\paragraph{Bipartite representations of directed graphs}
The first known (to us) bipartite representation of directed graphs was introduced by Gale (\cite{Gale}). We will use a slightly more elaborated version.
\begin{definition}
\emph{A partially paired bipartite graph} (PPBG for short) $(G,\F)$ is a bipartite graph $G(U,W)$ together with a set of pairwise disjoint vertex pairs $\F \subset U \times W$, such that $E(G) \cap \F = \emptyset$, that is, no vertex is adjacent to its $\F$-pair. We also call this object as an \emph{$\F$-free bipartite graph}.
 \end{definition}

\begin{definition}
A partially paired bipartite graph $(G,\F)$ represents the directed graph $\oG(X,E)$ on $|X|=|U|+|W|-|\F|$ nodes if there are injective functions $i_U : U \to X$ and $i_W : W \to X$ such that $i_U(u) = i_W(w)$ if and only if $(u,w)\in \F$, and if $(u,w) \in E(G) \Leftrightarrow \overrightarrow{i_U(u)i_W(w)} \in E(\oG)$.
\end{definition}
\noindent In words: when a vertex has zero out-degree in the original $\vec G$, then this vertex will not be represented in $U$. The analogous situation occurs, when the vertex has no in-degree. Of course, because our directed graph has no isolated vertex, therefore each original vertex is represented in at least one of the classes.  The edge set $\F$ represents the vertices with both positive in- and out-degree. Therefore $|\F|$ many vertices of the original directed graph will be represented in both $U$ and $W$. This explains the equation for $X$. The number of vertices in the original directed graph with zero out-degree is $m-|\F|.$ The analogous number for zero in-degree is $n-|\F|.$

\begin{lemma}\label{th:biject}
Every directed graph can be represented by a PPBG, and vice versa, every PPBG represents a directed graph that is unique up to relabeling of its nodes.
\end{lemma}

\begin{proof}
 For a given directed graph $\overrightarrow{G}(X,E)$, let $n$ denote the number of its nodes with positive out-degree and $m$ denote the number of its nodes with positive in-degree. Let $U = \{u_1,u_2,\dots,u_n\}$ and $W = \{w_1,w_2,\dots, w_m\}$, and fix injections $i_U : U \to X, i_W : W \to X$ such that the image of $i_U$ (respectively $i_W$) is exactly the set of nodes with positive out-degree (respectively in-degree). Let $\F = \{ (i_U^{-1}(x), i_W^{-1}(x)) | x \in X: d^+(x),d^-(x) > 0\}$. Finally let $G$ be a bipartite graph on $(U,W)$ with edge set given by $E(G) = \{ (i_U^{-1}(x), i_W^{-1}(y)) : \overrightarrow{xy} \in E(\oG)$.

 Vice versa, it is clear from the definitions, that an $\F$-free bipartite graph with vertex classes $U$ and $W$ represents a directed graph with $n+m-|\F|$ vertices. If the vertex $u\in U$ is covered by $\F$, then $u$ cannot have more than $m-1$ neighbors in $W.$ When we consider another $\F'$ with the same bidegree sequence, then this situation can change. Since we do not (and cannot) account for these alternatives, we will assume that
\[
\forall u\in U: d(u)<m \quad\text{and}\quad \forall w\in W:  d(w)<n.
\]
Since this bipartite representation of directed graphs is symmetric (does not depend which class represents the out-degrees and in-degrees), we will assume that $n\ge m.$

\end{proof}
\begin{remark}\label{th:Berg}
Applying the Gale-Ryser inequality \eqref{eq:ZZ} for an $\F$-free bipartite degree sequence can be done easily, but applying the Berger inequality \eqref{eq:Berger} from Theorem \ref{th:berger} would be more tricky, since when both sequences are in non-increasing order, then the $\F$ forbidden edges will connect vertices in different subscripts.

There are at least two strategies here for checking graphicality by Berger's theorem. On one hand side, applying the bijection among PPBGs and directed graphs described in Lemma \ref{th:biject}, we can apply Berger's theorem freely. On the other hand in the cases of this article when we want to check the graphicality of given PPBGs (like in Lemma \ref{th:biggest} or Theorem \ref{th:b-fully1}), we have the property that each forbidden edge connects vertex pairs with equal subscripts. Therefore the condition becomes necessary and sufficient again.
\end{remark}

\bigskip\noindent

\noindent
Now we define \textbf{$\F$-free bidegree regions}. For any formula involving the degree sequences and the forbidden partial matching $\F$, let us denote
\begin{equation}\label{eq:diregion3}
 \fD[\varphi] = \{ (C,D,\F) : \varphi(C,D,\F)\}.
\end{equation}
For a fixed matching $\F\subseteq U\times W$, define the corresponding
\textbf{basic $\F$-free region} by
\begin{align}\label{eq:DSR-fixed}
\fD(&\Sigma;n,\Gamma;m,\Delta;\F)=
\Big\{(C,D;\F): |C|=n,\ |D|=m,
\sum C=\sum D=\Sigma; \nonumber\\
&\hspace{35mm}
m-1\geq\gamma_1\geq c_i\geq\gamma_2
\text{ and }
n-1\geq\delta_1\geq d_j\geq\delta_2\Big\}.
\end{align}
The appropriate generalizations of the simple and very simple regions can also be defined:
\begin{definition}
A \textbf{simple} ($\F$-free bidegree) region is given as
\begin{align}\label{eq:DSR-S}
\fD(&\Sigma,n,\Gamma,m,\Delta;f)=\Big\{(C,D;\F): |C|=n,  |D|=m,\ \sum C=\sum  D=\Sigma,\ |\F|=f; \\
& m-1\ge \gamma_1\ge c_i\ge \gamma_2 \quad
\text{ and }\quad  n-1\ge \delta_1\ge d_j\ge \delta_2  \Big \}.\nonumber
\end{align}
A \textbf{very simple} region is given as
\[
\fD(n,\Gamma,m,\Delta;f)=\bigcup_{\substack{\max\{n\gamma_2,m\delta_2\}<\Sigma\\
                     \Sigma<\min\{n\gamma_1,m\delta_1\}}} \fD(\Sigma,n,\Gamma,m,\Delta;f).
\]
Equivalently,
\[
 \fD(\Sigma,n,\Gamma,m,\Delta;f)
 =\bigcup_{|\F|=f}\fD(\Sigma;n,\Gamma;m,\Delta;\F).
\]
\end{definition}

As in the bipartite case, a region may contain elements that are not graphic. We will be specifically interested in regions in which every element is graphic. Such regions are called \textbf{fully graphic}. For the sake of simplicity, we are omitting the $\F$-free and bipartite tags when referring to regions and graphicality as long as it does not lead to confusion.

\subsection{P-stability of partially paired bipartite sequences}\label{sec:stable}

\paragraph{P-stability}
The notion of $P$-stability of an infinite set of graphic degree sequences was introduced by Jerrum and Sinclair in \cite{JS90}. The notion plays a crucial role in complexity theory and graph theory. An analogous notion was introduced for bipartite and directed graphs only in 2022 by the paper \cite{P-stable}, where it was proved that the appropriate switch Markov chains are rapidly mixing on P-stable degree regions. Thanks to the notion of $\F$-free bidegree region, we can introduce these closely related notions from \cite{P-stable} in one go.

For a bidegree sequence $(C,D)$, let $\G(C,D,\F)$ denote the set of all $\F$-free bipartite realizations of $(C,D)$ (and let this set be empty when no such realization exists).  If $(C,D)$ has an $\F$-free realization, define its $\F$-free \textbf{perturbation ratio} by
\begin{equation}\label{eq:localP-F}
\pppp{(C,D,\F)}=\sum_{\substack{1\le i\le n\\1\le j\le m}}
\frac{\left|\G(C+e_i,D+e_j,\F)\right|}{\left|\G(C,D,\F)\right|},
\end{equation}
where the vector $e_i$  is the $i$th  unit vector of the appropriate dimension. Since $(C,D;\F)$ is $\F$-free graphic, therefore the denominator is non-zero. The operation $(C,D)\mapsto (C+e_i,D+e_j)$  is called  a \textbf{perturbation operation} on the bidegree sequence. The convention $\G(C',D',\F)=\varnothing$ applies whenever the perturbed pair is not a $\F$-free graphic bidegree sequence.  Notice that the perturbation increases both degree sums by one and therefore preserves
their equality.

An infinite region $\fB$ is called {\bf $P$-stable } if there is a polynomial $p(x)$ such that
\begin{equation}\label{eq:pstab}
\forall (C,D,\F) \in \fB: \mbox{ $(C,D,\F)$ is $\F$-free graphic} \Rightarrow \pppp{(C,D,\F)}\le p(|C|+|D|).
\end{equation}
When the edge set $\F=\emptyset$, then this definition coincides with the definition of $P$-stability of bipartite degree sequences, defined in \cite{P-stable}.

The notion can be extended meaningfully for finite regions by fixing the polynomial $p(x)$ instead of just assuming its existence. Thus, we will say that a region is $p$-stable when it satisfies \eqref{eq:pstab} for a concrete polynomial $p = p(x)$. Let us emphasize that we do not require    every element of a $P$-stable family to be graphic.

All these definitions, notions and notations are natural analogs for degree sequence regions, introduced in \cite{fully}.

\section{A general method to prove that an $\F$-free region is fully graphic}\label{sec:char}
\noindent
In this section we will describe a general approach to find fully $\F$-free graphic regions. The main benefit of these results lies in Theorem  \ref{th:stable}, which proves, that all fully $\F$-free graphic regions are also $(nm)^5$-stable, therefore the corresponding switch Markov chain is rapidly mixing on them, due to \cite{P-stable}.

\paragraph{A LEG type characterization.}
In the paper \cite{fully} a straightforward  sufficient and necessary condition (\cite[Theorem 2.5]{fully}) was introduced to decide whether a simple region is fully graphic.\footnote{The name of the notion was drawn from the expression \textbf{least Erd\H{o}s--Gallai degree sequence} referring the seminal Erd\H{o}s-Gallai theorem. While for bipartite and directed degree sequences other results provide necessary and sufficient conditions, we keep the name LEG.}
Here we will extend this approach for $\F$-free regions.

\begin{definition}\label{def:LEG}
Recall, that  $\gamma_1>\gamma_2>0$ and $\delta_1>\delta_2>0.$  Let us fix the parameter $f$ between $0$ and $\min(n,m).$ Define the bidegree sequence \textbf{LEG} as follows:
$$
\lbd(\Sigma,n,\Gamma;m,\Delta;f)=(\bar C,\bar D,\F_0)
$$
where
\begin{align}\label{eq:LEG}	
\bar C=\underbrace{\gamma_1,\ldots,\gamma_1}_{\alpha}, \gamma, &\underbrace{\gamma_2, \ldots, \gamma_2}_{(n-1)-\alpha} \quad\text{and}\quad
\bar D=\underbrace{\delta_1,\ldots,\delta_1}_{\beta}, \delta, \underbrace{\delta_2, \ldots, \delta_2}_{(m-1)-\beta}; \\
\F_0&=\{u_1w_1,u_2w_2,\ldots,u_fw_f\}.
\end{align}
We can easily compute that
\begin{equation}\label{eq:alfa}
\alpha=\left\lfloor \frac{\Sigma -n\cdot \gamma_2} {\gamma_1-\gamma_2}\right \rfloor \quad\text{and}\quad
\beta=\left\lfloor \frac{\Sigma -m\cdot \delta_2} {\delta_1-\delta_2}\right \rfloor;
\end{equation}
furthermore
\begin{equation}\label{eq:alfa2}
\gamma=\Sigma-[\alpha\gamma_1+(n-1-\alpha)\gamma_2]\quad\text{and}\quad
\delta=\Sigma-[\beta\delta_1+(m-1-\beta)\delta_2].
\end{equation}
Let's remark it is possible that $\gamma=\gamma_2$ or $\delta=\delta_2.$
\end{definition}

\begin{remark}\label{th:alfa}
Berger's inequalities show that this LEG is automatically graphic whenever
$\alpha>\delta_1$ or $\alpha<\delta_2$.  The analogous statements hold with
$\beta,\gamma_1,\gamma_2$ in place of
$\alpha,\delta_1,\delta_2$.  The equalities cannot be included in these
automatic cases.  For example, if
\[
 n=m=3,\quad f=2,\quad
 \gamma_1=\delta_1=2,\quad
 \gamma_2=\delta_2=1,\quad \Sigma=5,
\]
then $\alpha=\beta=2$ and
\[
 \bar C=\bar D=(2,2,1),\qquad
 \F_0=\{u_1w_1,u_2w_2\}.
\]
This triple has no avoiding realization: the degrees of $u_1$ and $u_2$
force both of them to be adjacent to $w_3$.
Thus the remaining parameter range is
\begin{equation}\label{eq:alfa1}
\delta_1\geq\alpha\geq\delta_2
\qquad\text{and}\qquad
\gamma_1\geq\beta\geq\gamma_2.
\end{equation}
\end{remark}

\begin{theorem}\label{th:DI-LEG}
The simple $\F$-free region $\fD(\Sigma,n,\Gamma,m,\Delta;f)$ is fully graphic if and only if the bidegree sequence $\lbb(\Sigma;n,\Gamma;m,\Delta;f)$ is graphic.
\end{theorem}
Before we turn to proving this statement we introduce an operation on $\F$-free bipartite graphs. Let $G$ be an $\F$-free realization of the bidegree sequence $(C,D)$, and let $x\in U, y,z\in W$ be distinct vertices such that $xy$ is an edge and $xz$ is a non-edge while $xz\not\in \F$. The $\F$-free \textbf{hinge-flip} operation $xy \rightarrow xz$ deletes the edge $xy$ from $G$ and adds the edge $xz$ to obtain a new, different graph $G'$, and their bidegree sequences are also different: while $C$ does not change, but $D'\ne D.$ The roles of $U$ and $W$ are reversible. The operation was already introduced in \cite{JS89} as a \textbf{Type 0 transition}.

We first record the comparison between the canonical forbidden matching and
an arbitrary forbidden matching of the same size.
\begin{lemma}\label{lem:canonical-matching}
Let
\[
 C=(c_1,\ldots,c_n),\qquad D=(d_1,\ldots,d_m)
\]
be non-increasing bidegree sequences on
$U=\{u_1,\ldots,u_n\}$ and $W=\{w_1,\ldots,w_m\}$.  Fix
$0\leq f\leq\min\{n,m\}$ and put
\[
 \F_0=\{u_iw_i:1\leq i\leq f\}.
\]
If $(C,D,\F_0)$ is graphic, then $(C,D,\F)$ is graphic for every matching
$\F\subseteq U\times W$ with $|\F|=f$.
\end{lemma}

\begin{proof}
This statement is an easy consequence of Remark \ref{rem:berger_order}, but for completeness, here we give a detailed proof of it as well. Fix such a matching $\F$, and put $N=n+m-f$.  In the directed representation of the PPBG, order the $N$ vertices so that the first $n$ are the images of $u_1,\ldots,u_n$ and the remaining $m-f$ are the images of the unpaired vertices of $W$.  The resulting out-degree sequence is
\[
 A=(c_1,\ldots,c_n,
       \underbrace{0,\ldots,0}_{m-f}),
\]
which is non-increasing.  Let $B^{\F}=(b_1,\ldots,b_N)$ be the corresponding
in-degree sequence.  For $i\leq n$, we have $b_i=d_j$ if
$(u_i,w_j)\in\F$, and $b_i=0$ if $u_i$ is unpaired.  The remaining
coordinates of $B^{\F}$ are the $d_j$ corresponding to the unpaired
vertices $w_j$.

For $k\leq n$, the left-hand side of Berger's inequality for
$(A,B^{\F})$ is
\begin{align*}
 L_{\F}(k)
 &=\sum_{r=1}^{N}\min\{b_r,k\}
       -|\{r\leq k:b_r\geq k\}|\\
 &=\sum_{j=1}^{m}\min\{d_j,k\}-P_{\F}(k),
\end{align*}
where
\[
 P_{\F}(k)=|\{(u_i,w_j)\in\F:i\leq k,\ d_j\geq k\}|.
\]
Put $q_k=|\{j:d_j\geq k\}|$.  Since $\F$ is a matching,
\[
 P_{\F}(k)\leq\min\{f,k,q_k\}.
\]
Since $D$ is non-increasing, the canonical matching satisfies
\[
 P_{\F_0}(k)
 =|\{i:i\leq f,\ i\leq k,\ d_i\geq k\}|
 =\min\{f,k,q_k\}.
\]
Therefore
\[
 L_{\F}(k)\geq L_{\F_0}(k)
 \qquad (k\leq n).
\]

The out-degree sequence $A$ does not depend on the matching.  Since
$(C,D,\F_0)$ is graphic, Berger's theorem gives
\[
 L_{\F_0}(k)\geq\sum_{i=1}^{k}c_i
\]
at every index $k\leq n$ required by that theorem.  The comparison above
gives the same inequality for $\F$.  If $N>n$, there are no further drop
indices before $N$, and the final inequality at $k=N$ is equality: every
in-degree is at most $n\leq N-1$, so both sides equal $\Sigma$.  If $N=n$,
the final index is already covered above; moreover, $m=f$, so every $w_j$ is
paired and graphicality of $(C,D,\F_0)$ implies $d_j\leq n-1$.  Hence all
inequalities required by Berger's theorem hold for $(A,B^{\F})$, and
$(C,D,\F)$ is graphic.
\end{proof}

\begin{proof}[Proof of the Theorem \ref{th:DI-LEG}]
The ``only if'' direction follows because
$(\bar C,\bar D,\F_0)$ belongs to the region.  For the converse, list the
elements of $C$ and $D$ in non-increasing order, transporting $\F$ under the
corresponding relabelings.  The transported forbidden set is still a matching
of size $f$.

We consider the \textbf{lexicographic decreasing order} $\succ$ on all pairs $(C,D)$ in $\fD$, so  $(C,D) \succ (C',D')$ iff $C \succ C'$ or $C=C'$ and $D\succ D'$. Under this ordering the largest element of $\fD$ is $(\bar C,\bar D)$.

Fix a forbidden edge set $\F$ with  $f$ edges, and fix the sequence $D^*$. As always, all sequences are in non-increasing order. Take all possible sequences $C$ such that
$$(C,D^*,\F) \in \fD(\Sigma,n,\Gamma;m,\Delta;\F).$$
Denote the set of these sequences by $\mC$. Clearly, $\bar C$ is the lexicographically  largest element in $\mC$.
\begin{lemma}
Assume that $(\bar C,D^*,\F)$ is graphic (there exists an $\F$-free bipartite realization). Then each pair in $(\mC,D^*,\F)$ is also graphic.
\end{lemma}

\begin{proof}
Consider the largest $C'\in \mC$ such that $(C',D^*,\F)$ is not graphic.
By our assumption $\bar C\succ C'$.  Let $i$ be the first index such that
$\bar c_i>c'_i$, and let $j_0>i$ be the first index such that
$\bar c_{j_0}<c'_{j_0}$.  Such indices exist because the two sequences have
the same sum.  We have $i\leq\alpha+1\leq j_0$.

Let $j$ be the last index in the constant block of $C'$ containing $j_0$;
that is, $c'_j=c'_{j_0}$ and either $j=n$ or $c'_j>c'_{j+1}$.  This $j$ is
still a surplus index.  Indeed, if $j_0>\alpha+1$, then
$\bar c_j=\bar c_{j_0}=\gamma_2<c'_{j_0}=c'_j$.  If
$j_0=\alpha+1$ and $j>j_0$, then
$\bar c_j=\gamma_2\leq\bar c_{j_0}<c'_{j_0}=c'_j$; the case $j=j_0$ is
immediate.

Set $C^\star=C'+e_i-e_j$.  The choice of the first deficit index permits
the increase at $i$: if $i>1$, then the coordinates before $i$ agree with
$\bar C$, and hence
$c'_{i-1}=\bar c_{i-1}\geq\bar c_i\geq c'_i+1$.
The choice of the last index in its constant block permits the decrease at
$j$: if $j<n$, then $c'_j-1\geq c'_{j+1}$.  Hence $C^\star$ is
non-increasing.  Moreover, the deficit and surplus inequalities show that
all its entries remain between $\gamma_2$ and $\gamma_1$.  Thus
$C^\star\in\mC$ and $C^\star\succ C'$.  By the maximality of $C'$, the
bidegree sequence $(C^\star,D^*)$ is graphic; let $G^\star$ be a realization.

Since $i<j$ and $C'$ is non-increasing,
\[
 c^\star_i=c'_i+1\geq c'_j+1=c^\star_j+2.
\]
Consequently there is a vertex $w\in W$ such that
$u_iw\in E(G^\star)$, $u_jw\notin E(G^\star)$, and $u_jw\notin\F$.
(The degree difference is at least two, while at most one pair incident with
$u_j$ is forbidden.)  The hinge flip $u_iw\rightarrow u_jw$ gives a
realization of $(C',D^*)$, a contradiction.
\end{proof}
\noindent Our next task is to prove that the triplet $(\bar C,D,\F)$ is also graphic for each feasible $D.$ Take all possible sequence $D$ such that
$$(\bar C,D) \in \fD(\Sigma,n,\Gamma;m,\Delta;\F).$$
Denote the set of these sequences by $\mD.$ Clearly, $\bar D$ is the lexicographically largest element in $\mD.$
\begin{lemma}\label{th:biggest}
If $(\bar C,\bar D,\F)$ is graphic, then for each $D\in\mD$ the triple
$(\bar C,D,\F)$ is also graphic.
\end{lemma}
\begin{proof}
Repeat the preceding proof with the two vertex classes exchanged, using the
last index in the relevant constant block of $D'$ for the downward unit
transfer.
\end{proof}
It remains to pass from the canonical forbidden matching to the arbitrary
matching fixed above.
\begin{lemma}\label{lem:fdependence}
If $\lbd(\Sigma,n,\Gamma;m,\Delta;f)=(\bar C,\bar D,\F_0)$ is graphic,
then $(\bar C,\bar D,\F)$ is graphic for every matching $\F$ of size $f$.
\end{lemma}
\begin{proof}
Apply Lemma~\ref{lem:canonical-matching} with $C=\bar C$ and $D=\bar D$.
\end{proof}
For the fixed matching $\F$, the last lemma first gives the graphicality of
$(\bar C,\bar D,\F)$.  Lemma~\ref{th:biggest} then gives the graphicality of
$(\bar C,D,\F)$ for every feasible $D$, and the first smoothing lemma gives
the graphicality of $(C,D,\F)$ for every feasible $C$.  Since $\F$ was
arbitrary, the whole region is fully graphic.  This finishes the proof of
Theorem~\ref{th:DI-LEG}.
\end{proof}

\section{Every fully $\F$-free graphic simple region is P-stable} \label{sec:forbid}
\begin{theorem}\label{tm:main}\label{th:stable}
Let $U=\{u_1,\ldots,u_n\}$ and $W=\{w_1,\ldots,w_m\}$ be disjoint
vertex classes, where $n\ge m$, and let $\F\subseteq U\times W$ be a
matching, which can be empty.  If the basic region
\[
\fD=\fD(\Sigma;n,\Gamma;m,\Delta;\F)
\]
is fully $\F$-free graphic, then
\[
\pppp{(C,D,\F)}\le (nm)^5
\]
for every $(C,D;\F)\in\fD$.
\end{theorem}
\noindent Of course the statement extends readily for any union of such basic regions, that is for every fully $\F$-free graphic simple and very simple region.

\medskip\noindent
We first introduce some additional terminology.  Let $G=(U,W,E_G)$ be an
$\F$-free bipartite graph; that is, $E_G\cap \F=\emptyset$.  Set
\begin{displaymath}
\C=(U\times W)\setminus \F,\qquad
\aE_G=\C\setminus E_G.
\end{displaymath}
We call the elements of $\C$ \emph{chords} and the elements of $\aE_G$
\emph{non-edges}.  Thus the chords are precisely the vertex pairs that may
occur as edges in an $\F$-free realization, and throughout this section a
non-edge always means a chord which is not an edge of the graph under
discussion.  Equivalently, for each such graph $G$,
$\C=E_G\uplus\aE_G$ is the disjoint union of its edges and non-edges.

We also need several definitions analogous to those for simple graphs in
\cite[Section 2]{Simple-fully}.
\begin{definition}\label{df:spec-trails}
Let $G$ be an $\F$-free bipartite graph on $(U,W)$, and let $u_p\in U$
and $w_q\in W$.  An alternating trail of edges and non-edges from $u_p$ to $w_q$ necessarily has odd length.  Such a trail is \textbf{edge-abundant} if its first element is an edge.  Thus an edge-abundant trail contains one more edge than non-edges. An edge-abundant trail of length at most $k$ is a
\textbf{$k$-witness trail}.  When the length bound is understood, we simply call it a witness trail.
\end{definition}
We will use the elementary fact that if there is an alternating \emph{trail} between two vertices from different classes in a bipartite graph, then this trail contains an alternating \emph{path} between the same vertices as well.  Indeed, if a vertex occurs twice on the trail, then its two occurrences have the same parity in the vertex sequence.  The intervening closed subtrail therefore has even length and may be deleted without changing the alternation or the status of the first and last pairs.  Repeating this operation removes all repeated vertices.  In particular, every edge-abundant alternating trail contains an edge-abundant alternating path of no greater length.

The following $\F$-free bipartite version of \cite[Proposition 2.2]{Simple-fully} will be useful.  We include the short proof to verify that no forbidden pair occurs on the resulting trail.
\begin{proposition}\label{pr:find_trail}
Let $U$ and $W$ be disjoint vertex classes, let $\F\subseteq U\times W$ be a matching, and suppose that we have two realizations
\begin{displaymath}
H_0\in \G(C,D,\F) \qquad\text{and} \qquad H_1\in\G(C+e_p,D+e_q,\F).
\end{displaymath}
Then the symmetric difference $E(H_0)\mathbin\triangle E(H_1)$ contains
an edge-abundant alternating trail from $u_p$ to $w_q$, with its pairs
alternating between $E(H_1)\setminus E(H_0)$ and $E(H_0)\setminus E(H_1)$.  In particular, $H_1$ has an edge-abundant alternating path from $u_p$ to $w_q$.
\end{proposition}
\begin{proof}
Construct an auxiliary directed graph on $U\cup W$ as follows.  Orient
every pair in $E(H_1)\setminus E(H_0)$ from $U$ to $W$, and orient every
pair in $E(H_0)\setminus E(H_1)$ from $W$ to $U$.  At every vertex other
than $u_p$ and $w_q$, the in-degree equals the out-degree.  At $u_p$ the
out-degree exceeds the in-degree by one, while at $w_q$ the in-degree
exceeds the out-degree by one.

Add one auxiliary arc from $w_q$ to $u_p$.  The in-degree and out-degree
are now equal at every vertex, so the weak component containing this arc
has an Euler circuit.  Start the circuit with the auxiliary arc and then
delete that arc.  What remains is a directed trail from $u_p$ to $w_q$.
By the orientations, its pairs alternate between
$E(H_1)\setminus E(H_0)$ and $E(H_0)\setminus E(H_1)$, beginning and
ending with a pair in $E(H_1)\setminus E(H_0)$.

Every pair in $E(H_0)\setminus E(H_1)$ is outside $\F$, because $H_0$
is $\F$-free, and it is not an edge of $H_1$.  It is therefore a non-edge of $H_1$.  Thus the trail is edge-abundant and alternating in $H_1$.  The trail-to-path reduction above gives the last assertion.
\end{proof}

\medskip\noindent
The next theorem is the main step in the proof of Theorem~\ref{tm:main}.
\begin{theorem}\label{tm:main-tree}
Let $U=\{u_1,\ldots,u_n\}$ and $W=\{w_1,\ldots,w_m\}$ be disjoint
vertex classes, let $\F\subseteq U\times W$ be a matching, and let
$\fD=\fD(\Sigma;n,\Gamma;m,\Delta;\F)$.  Fix $(C,D;\F)\in \fD$.  Suppose that $G$ is an $\F$-free bipartite graph with
bidegree sequence
\begin{displaymath}
\bd(G)=(C+e_p,D+e_q)
\end{displaymath}
for some $1\le p\le n$ and $1\le q\le m$.  If there is no
edge-abundant alternating path of length at most $9$ from $u_p$ to $w_q$, then
$\fD$ is not fully $\F$-free graphic.
\end{theorem}

Throughout this section $G$ is the $\F$-free bipartite graph from
Theorem~\ref{tm:main-tree}, with vertex classes $U$ and $W$, and
$\C$ and $\aE_G$ are understood as above.  If $A$ is a subset of one
vertex class, then $\Gamma_G(A)$ denotes its neighborhood in the other
class, while
\begin{displaymath}
\Nec_G(A)=\{v:\text{ there is }a\in A\text{ such that }av\in \aE_G\}
\end{displaymath}
denotes its anti-neighborhood through non-edges.

We shall repeatedly subtract the extra unit at the distinguished
vertices $u_p$ and $w_q$.  If $H$ is an $\F$-free bipartite graph with
$\bd(H)=(A,B)$ and $A_p,B_q\ge 1$, then we write
\begin{displaymath}
\bd(H)-(e_p,e_q)=(A-e_p,B-e_q).
\end{displaymath}
Thus, if $\bd(H)=(C'+e_p,D'+e_q)$, then
$\bd(H)-(e_p,e_q)=(C',D')$.

We prove Theorem~\ref{tm:main-tree} by producing a bidegree sequence of
the form $\bd(G')-(e_p,e_q)$ in the region that has no
$\F$-free realization.  The first step is to isolate a
\textbf{blocking configuration} that certifies non-realizability.
Definition~\ref{def:bip-blocking} describes such a partition for a
graph $G'$ with bidegree sequence $(C'+e_p,D'+e_q)$, and
Lemma~\ref{lem:bip-blocking-nonrealizable} shows that
$(C',D')=\bd(G')-(e_p,e_q)$ has no $\F$-free realization.

Thus, to prove Theorem~\ref{tm:main-tree}, it will be enough to find a
graph $G'$ with a bipartite blocking configuration such that
$(C',D')$ still belongs to the prescribed region.  The original graph
$G$ need not have the required partition.  We therefore modify $G$ by
small edge moves that keep $\bd(G')-(e_p,e_q)$ inside the
region.  During this modification, we no longer require the absence of
short edge-abundant alternating paths; the modified graph may contain
such a path of length at most $9$ between $u_p$ and $w_q$.  We stop when
the resulting graph has the partition described in
Definition~\ref{def:bip-blocking}.

\begin{definition}\label{def:bip-blocking}
Let $G'$ be an $\F$-free bipartite graph such that
\begin{displaymath}
\bd(G')=(C'+e_p,D'+e_q).
\end{displaymath}
We say that $G'$ has a \textbf{bipartite blocking configuration}
with respect to $u_p$ and $w_q$ if there are partitions
\begin{displaymath}
U=\{u_p\}\uplus X'_U\uplus Y'_U\uplus R'_U,\qquad
W=\{w_q\}\uplus X'_W\uplus Y'_W\uplus R'_W
\end{displaymath}
such that the following hold:
\begin{enumerate}[(a)]
\item every chord in $X'_U\times X'_W$ is an edge of $G'$;
\item every chord in $(X'_U\times R'_W)\cup(R'_U\times X'_W)$ is an edge of $G'$;
\item every chord in $Y'_U\times Y'_W$ is a non-edge of $G'$;
\item every chord in $(Y'_U\times R'_W)\cup(R'_U\times Y'_W)$ is a non-edge of $G'$;
\item $\Gamma_{G'}(u_p)\subseteq X'_W$ and $\Gamma_{G'}(w_q)\subseteq X'_U$.
\end{enumerate}
\end{definition}

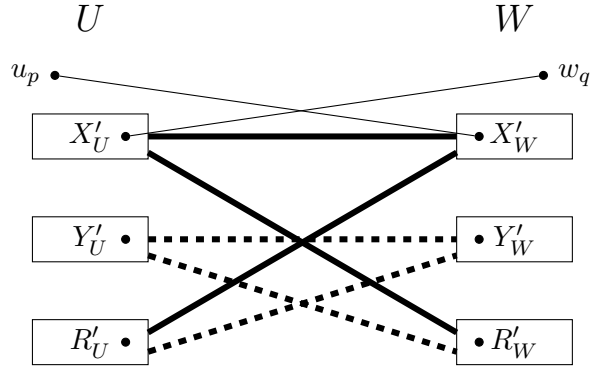
\begin{figure}[H]
\begin{center}
\begin{tikzpicture}[scale=.85]
\tikzstyle{vertex}=[draw,circle,fill=black,minimum size=3,inner sep=0]
\def\uBoxLeft{-.7}
\def\uBoxRight{1.1}
\def\uBoxCenter{.2}
\def\uVertexX{.75}
\def\wBoxLeft{5.9}
\def\wBoxRight{7.7}
\def\wBoxCenter{6.8}
\def\wVertexX{6.25}
\def\xBoxBottom{4.65}
\def\xBoxCenter{5}
\def\xBoxTop{5.35}
\def\yBoxBottom{3.05}
\def\yBoxCenter{3.4}
\def\yBoxTop{3.75}
\def\rBoxBottom{1.45}
\def\rBoxCenter{1.8}
\def\rBoxTop{2.15}
\node at (.2,6.85) {\Large $U$};
\node at (6.8,6.85) {\Large $W$};

\node[vertex] (up) at (-.35,5.95) [label=west:{$u_p$}] {};
\node[vertex] (wq) at (7.25,5.95) [label=east:{$w_q$}] {};

\draw (\uBoxLeft,\xBoxBottom) rectangle (\uBoxRight,\xBoxTop);
\node at (\uBoxCenter,\xBoxCenter) {$X'_U$};
\node[vertex] (ku) at (\uVertexX,\xBoxCenter) {};
\draw (\wBoxLeft,\xBoxBottom) rectangle (\wBoxRight,\xBoxTop);
\node at (\wBoxCenter,\xBoxCenter) {$X'_W$};
\node[vertex] (kw) at (\wVertexX,\xBoxCenter) {};

\draw (\uBoxLeft,\yBoxBottom) rectangle (\uBoxRight,\yBoxTop);
\node at (\uBoxCenter,\yBoxCenter) {$Y'_U$};
\node[vertex] (yu) at (\uVertexX,\yBoxCenter) {};
\draw (\wBoxLeft,\yBoxBottom) rectangle (\wBoxRight,\yBoxTop);
\node at (\wBoxCenter,\yBoxCenter) {$Y'_W$};
\node[vertex] (yw) at (\wVertexX,\yBoxCenter) {};

\draw (\uBoxLeft,\rBoxBottom) rectangle (\uBoxRight,\rBoxTop);
\node at (\uBoxCenter,\rBoxCenter) {$R'_U$};
\node[vertex] (ru) at (\uVertexX,\rBoxCenter) {};
\draw (\wBoxLeft,\rBoxBottom) rectangle (\wBoxRight,\rBoxTop);
\node at (\wBoxCenter,\rBoxCenter) {$R'_W$};
\node[vertex] (rw) at (\wVertexX,\rBoxCenter) {};

\draw (up) -- (kw);
\draw (wq) -- (ku);

\draw[line width=.8mm] (\uBoxRight,\xBoxCenter) -- (\wBoxLeft,\xBoxCenter);
\draw[line width=.8mm] (\uBoxRight,4.75) -- (\wBoxLeft,1.95);
\draw[line width=.8mm] (\uBoxRight,1.95) -- (\wBoxLeft,4.75);

\draw[line width=.8mm,dashed] (\uBoxRight,\yBoxCenter) -- (\wBoxLeft,\yBoxCenter);
\draw[line width=.8mm,dashed] (\uBoxRight,3.15) -- (\wBoxLeft,1.65);
\draw[line width=.8mm,dashed] (\uBoxRight,1.65) -- (\wBoxLeft,3.15);
\end{tikzpicture}
\end{center}
\caption{A bipartite blocking configuration.  Thick solid lines indicate that every chord in the corresponding pair of parts is an edge of $G'$, while thick dashed lines indicate that every such chord is a non-edge.  The only possible neighbors of $u_p$ and $w_q$ lie in $X'_W$ and $X'_U$, respectively.}
\label{fig:bip-blocking}
\end{figure}

\begin{lemma}\label{lem:bip-blocking-nonrealizable}
If $G'$ has a bipartite blocking configuration with respect to
$u_p$ and $w_q$, then the bidegree sequence
\begin{displaymath}
\bd(G')-(e_p,e_q)
\end{displaymath}
has no $\F$-free realization.
\end{lemma}

\begin{proof}
Suppose, toward a contradiction, that $H$ is an $\F$-free realization
of $\bd(G')-(e_p,e_q)$.  In the symmetric difference $E(G')\triangle E(H)$, every
vertex other than $u_p$ and $w_q$ is incident with the same number of
edges from $E(G')\setminus E(H)$ and $E(H)\setminus E(G')$.  At each of
$u_p$ and $w_q$, the former number exceeds the latter by one.  To see
that the two exceptional vertices are joined, orient the pairs in
$E(G')\setminus E(H)$ from $U$ to $W$ and those in
$E(H)\setminus E(G')$ from $W$ to $U$.  Every vertex is balanced except
that $u_p$ has one extra outgoing arc and $w_q$ has one extra incoming
arc.  Hence the directed-trail decomposition contains a trail from
$u_p$ to $w_q$.  In the original symmetric difference, this is an
alternating trail whose first and last edges belong to
$E(G')\setminus E(H)$.

Let this trail be
\begin{displaymath}
u_p=a_0,a_1,\ldots,a_{2s},a_{2s+1}=w_q,
\end{displaymath}
where the edges with odd indices belong to $E(G')\setminus E(H)$ and
the edges with even indices belong to $E(H)\setminus E(G')$.

\begin{claim}
Every $a_{2i+1}$ belongs to $X'_W$, and every $a_{2i}$ belongs to $\{u_p\}\cup Y'_U$.
\end{claim}

We prove the claim by induction.  The assertion for $a_1$ follows from
$\Gamma_{G'}(u_p)\subseteq X'_W$.  Suppose that
$a_{2i+1}\in X'_W$.  The pair $a_{2i+1}a_{2i+2}$ belongs to
$E(H)\setminus E(G')$.  Since $H$ is $\F$-free, this pair is a chord;
since it is not an edge of $G'$, it is a non-edge of $G'$.
Conditions (a) and (b) imply that $a_{2i+2}\notin X'_U\cup R'_U$, so
$a_{2i+2}\in\{u_p\}\cup Y'_U$.  Conversely, suppose that
$a_{2i}\in\{u_p\}\cup Y'_U$.  The next pair is an edge of $G'$.
From $u_p$, it leads to $X'_W$ by (e).  From $Y'_U$, it cannot lead to
$Y'_W\cup R'_W$ by (c) and (d), nor can it lead to $w_q$, because
$\Gamma_{G'}(w_q)\subseteq X'_U$.  It must therefore lead to $X'_W$.

In particular, the induction gives $a_{2s+1}\in X'_W$, contradicting
$a_{2s+1}=w_q$ and the partition
\begin{displaymath}
W=\{w_q\}\uplus X'_W\uplus Y'_W\uplus R'_W.
\end{displaymath}
Therefore $\bd(G')-(e_p,e_q)$ has no $\F$-free realization.
\end{proof}

\begin{proof}[Proof of Theorem~\ref{tm:main-tree}]
We work under the hypotheses of the theorem; in particular, there is no
edge-abundant alternating path of length at most $9$ from $u_p$ to
$w_q$ in the original graph $G$.

We now carry out the strategy described before
Definition~\ref{def:bip-blocking}.  Starting from $G$, we construct
another $\F$-free bipartite graph $G_L$ such that
\begin{displaymath}
\bd(G_L)-(e_p,e_q)=(C',D')\in\fD(\Sigma;n,\Gamma;m,\Delta;\F),
\end{displaymath}
and such that $G_L$ has a bipartite blocking configuration with respect to
$u_p$ and $w_q$.  Lemma~\ref{lem:bip-blocking-nonrealizable} will then
show that $(C',D')$ has no $\F$-free realization, proving that
$\fD(\Sigma;n,\Gamma;m,\Delta;\F)$ is not fully $\F$-free graphic.

The notation anticipates the final blocking configuration: at the end
of the construction we will take
\begin{displaymath}
X'_U=X^U_1\cup X^U_2\cup X^U_3,\qquad
Y'_U=Y^U_1\cup Y^U_2,\qquad
R'_U=R^U,
\end{displaymath}
and similarly
\begin{displaymath}
X'_W=X^W_1\cup X^W_2\cup X^W_3,\qquad
Y'_W=Y^W_1\cup Y^W_2,\qquad
R'_W=R^W.
\end{displaymath}

We now begin the actual construction of these sets from the original
graph $G$.
Define
\begin{displaymath}
X^U_1=\Gamma_G(w_q),\qquad X^W_1=\Gamma_G(u_p),
\end{displaymath}
and then record the vertices that can be reached next through a
non-edge:
\begin{displaymath}
Y^U_1=\Nec_G(X^W_1),\qquad Y^W_1=\Nec_G(X^U_1).
\end{displaymath}
Thus every vertex of $Y^U_1$ has a non-edge to some vertex
of $X^W_1$, and symmetrically every vertex of $Y^W_1$ has a
non-edge to some vertex of $X^U_1$.  Next, put
\begin{displaymath}
X^U_2=\Gamma_G(Y^W_1)\setminus X^U_1,\qquad
X^W_2=\Gamma_G(Y^U_1)\setminus X^W_1,
\end{displaymath}
and, for brevity, set
\begin{displaymath}
X^U_{1,2}=X^U_1\cup X^U_2,\qquad X^W_{1,2}=X^W_1\cup X^W_2.
\end{displaymath}
Finally, let
\begin{displaymath}
S^U=U\setminus(\{u_p\}\cup X^U_{1,2}\cup Y^U_1),\qquad
S^W=W\setminus(\{w_q\}\cup X^W_{1,2}\cup Y^W_1).
\end{displaymath}
Within the remaining sets, let
\begin{displaymath}
Y^U_2=S^U\cap\Nec_G(X^W_{1,2}),\qquad
Y^W_2=S^W\cap\Nec_G(X^U_{1,2}),
\end{displaymath}
and put
\begin{displaymath}
Z^U=S^U\setminus Y^U_2,\qquad Z^W=S^W\setminus Y^W_2.
\end{displaymath}

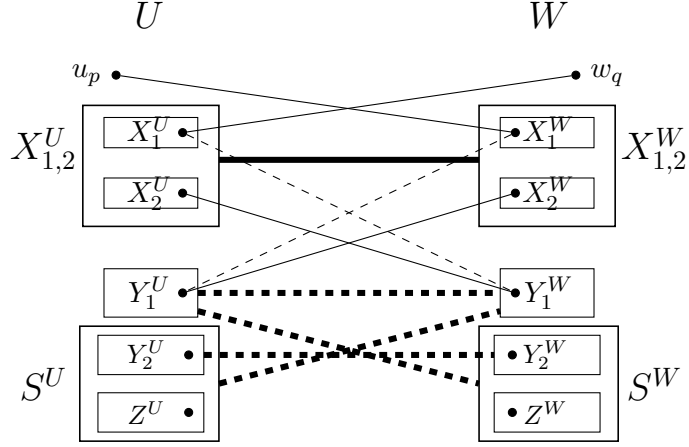
\begin{figure}[t]
\begin{center}
\begin{tikzpicture}[scale=.8]
\tikzstyle{vertex}=[draw,circle,fill=black,minimum size=3,inner sep=0]
\def\uOuterLeft{-.9}
\def\uOuterRight{1.35}
\def\uOuterLabelX{-1.55}
\def\uPartLeft{-.55}
\def\uPartRight{1}
\def\uPartCenter{.2}
\def\uPartVertexX{.75}
\def\uSOuterLeft{-.95}
\def\uSOuterRight{1.35}
\def\uSPartLeft{-.65}
\def\uSPartRight{1.1}
\def\uSPartCenter{.15}
\def\uSVertexX{.85}
\def\wOuterLeft{5.65}
\def\wOuterRight{7.9}
\def\wOuterLabelX{8.55}
\def\wPartLeft{6}
\def\wPartRight{7.55}
\def\wPartCenter{6.8}
\def\wPartVertexX{6.25}
\def\wSOuterLeft{5.65}
\def\wSOuterRight{7.95}
\def\wSPartLeft{5.9}
\def\wSPartRight{7.65}
\def\wSPartCenter{6.75}
\def\wSVertexX{6.2}
\def\xOuterBottom{3.75}
\def\xOuterTop{5.75}
\def\xOuterLabelY{5.00}
\def\xOneBottom{5.05}
\def\xOneTop{5.55}
\def\xOneCenter{5.3}
\def\xTwoBottom{4.05}
\def\xTwoTop{4.55}
\def\xTwoCenter{4.3}
\def\yOneBottom{2.25}
\def\yOneTop{3.05}
\def\yOneCenter{2.65}
\def\sOuterBottom{.2}
\def\sOuterTop{2.1}
\def\sOuterLabelY{1.15}
\def\yTwoBottom{1.3}
\def\yTwoTop{1.95}
\def\yTwoCenter{1.625}
\def\zBottom{.35}
\def\zTop{1}
\def\zCenter{.675}
\def\betweenCenter{3.5}
\def\xEdgeLineY{4.85}
\def\xEdgeLabelY{5.15}
\def\yLowLineY{2.35}
\node at (\uPartCenter,7.25) {\Large $U$};
\node at (\wPartCenter,7.25) {\Large $W$};

\node[vertex] (up) at (-.35,6.25) [label=west:{$u_p$}] {};
\node[vertex] (wq) at (7.25,6.25) [label=east:{$w_q$}] {};

\draw[line width=.6pt] (\uOuterLeft,\xOuterBottom) rectangle (\uOuterRight,\xOuterTop);
\node at (\uOuterLabelX,\xOuterLabelY) {\Large $X^U_{1,2}$};
\draw (\uPartLeft,\xOneBottom) rectangle (\uPartRight,\xOneTop);
\node at (\uPartCenter,\xOneCenter) {$X^U_1$};
\node[vertex] (xu) at (\uPartVertexX,\xOneCenter) {};
\draw (\uPartLeft,\xTwoBottom) rectangle (\uPartRight,\xTwoTop);
\node at (\uPartCenter,\xTwoCenter) {$X^U_2$};
\node[vertex] (zu) at (\uPartVertexX,\xTwoCenter) {};

\draw[line width=.6pt] (\wOuterLeft,\xOuterBottom) rectangle (\wOuterRight,\xOuterTop);
\node at (\wOuterLabelX,\xOuterLabelY) {\Large $X^W_{1,2}$};
\draw (\wPartLeft,\xOneBottom) rectangle (\wPartRight,\xOneTop);
\node at (\wPartCenter,\xOneCenter) {$X^W_1$};
\node[vertex] (xw) at (\wPartVertexX,\xOneCenter) {};
\draw (\wPartLeft,\xTwoBottom) rectangle (\wPartRight,\xTwoTop);
\node at (\wPartCenter,\xTwoCenter) {$X^W_2$};
\node[vertex] (zw) at (\wPartVertexX,\xTwoCenter) {};

\draw (\uPartLeft,\yOneBottom) rectangle (\uPartRight,\yOneTop);
\node at (\uPartCenter,\yOneCenter) {$Y^U_1$};
\node[vertex] (yu) at (\uPartVertexX,\yOneCenter) {};
\draw (\wPartLeft,\yOneBottom) rectangle (\wPartRight,\yOneTop);
\node at (\wPartCenter,\yOneCenter) {$Y^W_1$};
\node[vertex] (yw) at (\wPartVertexX,\yOneCenter) {};

\draw[line width=.6pt] (\uSOuterLeft,\sOuterBottom) rectangle (\uSOuterRight,\sOuterTop);
\node at (\uOuterLabelX,\sOuterLabelY) {\Large $S^U$};
\draw (\uSPartLeft,\yTwoBottom) rectangle (\uSPartRight,\yTwoTop);
\node at (\uSPartCenter,\yTwoCenter) {\small $Y^U_2$};
\node[vertex] (ruN) at (\uSVertexX,\yTwoCenter) {};
\draw (\uSPartLeft,\zBottom) rectangle (\uSPartRight,\zTop);
\node at (\uSPartCenter,\zCenter) {\small $Z^U$};
\node[vertex] (ruE) at (\uSVertexX,\zCenter) {};
\draw[line width=.6pt] (\wSOuterLeft,\sOuterBottom) rectangle (\wSOuterRight,\sOuterTop);
\node at (\wOuterLabelX,\sOuterLabelY) {\Large $S^W$};
\draw (\wSPartLeft,\yTwoBottom) rectangle (\wSPartRight,\yTwoTop);
\node at (\wSPartCenter,\yTwoCenter) {\small $Y^W_2$};
\node[vertex] (rwN) at (\wSVertexX,\yTwoCenter) {};
\draw (\wSPartLeft,\zBottom) rectangle (\wSPartRight,\zTop);
\node at (\wSPartCenter,\zCenter) {\small $Z^W$};
\node[vertex] (rwE) at (\wSVertexX,\zCenter) {};

\draw (up) -- (xw);
\draw (wq) -- (xu);
\draw[dashed] (xw) -- (yu);
\draw[dashed] (xu) -- (yw);
\draw (yu) -- (zw);
\draw (yw) -- (zu);

\draw[line width=.8mm] (\uOuterRight,\xEdgeLineY) -- (\wOuterLeft,\xEdgeLineY);
\draw[line width=.8mm,dashed] (\uPartRight,\yOneCenter) -- (\wPartLeft,\yOneCenter);
\draw[line width=.8mm,dashed] (\uPartRight,\yLowLineY) -- (\wSOuterLeft,\sOuterLabelY);
\draw[line width=.8mm,dashed] (\uSOuterRight,\sOuterLabelY) -- (\wPartLeft,\yLowLineY);
\draw[line width=.8mm,dashed] (\uSPartRight,\yTwoCenter) -- (\wSPartLeft,\yTwoCenter);
\end{tikzpicture}
\end{center}
\caption{The partition of the bipartite graph around $u_p$ and $w_q$.
Solid lines indicate edges, and dashed lines indicate non-edges; thick lines indicate that all corresponding chords have the same status.}
\label{fig:bipartite-layers}
\end{figure}

We first prove three consequences of the absence of short witnesses.

\smallskip
\begin{claim}\label{claim1}
There are no edges of $G$ in
\begin{displaymath}
Y^U_1\times Y^W_1,\qquad Y^U_1\times S^W,\qquad S^U\times Y^W_1.
\end{displaymath}
\end{claim}

\begin{proof}
Suppose first that $y^U_1y^W_1$ is an edge of $G$, where
$y^U_1\in Y^U_1$ and $y^W_1\in Y^W_1$.  By the definitions of $Y^U_1$
and $Y^W_1$, choose $x^W_1\in X^W_1$ and $x^U_1\in X^U_1$ such that
$x^W_1y^U_1,x^U_1y^W_1\in\aE_G$.  Since $u_px^W_1$ and $x^U_1w_q$ are
edges of $G$, the path
\begin{displaymath}
u_p-x^W_1\cdots y^U_1-y^W_1\cdots x^U_1-w_q
\end{displaymath}
is an edge-abundant alternating path of length $5$, a contradiction.
Next, if $y^U_1w$ were an edge with $y^U_1\in Y^U_1$ and $w\in S^W$,
then the definition of $X^W_2$ would imply $w\in X^W_2$, contradicting
the definition of $S^W$.  The proof for $S^U\times Y^W_1$ is symmetric.
\end{proof}

\smallskip
\begin{claim}\label{claim2}
Every chord in
\begin{displaymath}
X^U_{1,2}\times X^W_{1,2},\qquad X^U_1\times S^W,\qquad S^U\times X^W_1
\end{displaymath}
is an edge of $G$.
\end{claim}
\begin{proof}
We prove this by excluding non-edges.
First, if $x^U_1x^W_1\in\aE_G$ with $x^U_1\in X^U_1$ and
$x^W_1\in X^W_1$, then
\begin{displaymath}
u_p-x^W_1\cdots x^U_1-w_q
\end{displaymath}
is an edge-abundant alternating path of length $3$, a contradiction.
Next, suppose that $x^U_1w\in\aE_G$ with $x^U_1\in X^U_1$ and
$w\in X^W_2\cup S^W$.  Then $w\in Y^W_1$ by the definition of
$Y^W_1$.  This is impossible if $w\in S^W$.  If $w\in X^W_2$, then
some $y^U_1\in Y^U_1$ satisfies $y^U_1w\in E(G)$, contradicting
Claim~\ref{claim1}.  The symmetric argument excludes
non-edges in $(X^U_2\cup S^U)\times X^W_1$.
Finally, if $x^U_2x^W_2\in\aE_G$, choose $y^U_1\in Y^U_1$ and
$y^W_1\in Y^W_1$ such that $y^U_1x^W_2,x^U_2y^W_1\in E(G)$, and then
choose $x^W_1\in X^W_1$ and $x^U_1\in X^U_1$ such that
$x^W_1y^U_1,x^U_1y^W_1\in\aE_G$.  Then
\begin{displaymath}
u_p-x^W_1\cdots y^U_1-x^W_2\cdots x^U_2-y^W_1\cdots x^U_1-w_q
\end{displaymath}
is an edge-abundant alternating path of length $7$, a contradiction.
\end{proof}

\smallskip
\begin{claim}\label{claim3}
There are no edges of $G$ in $Y^U_2\times Y^W_2$, and every chord in
\begin{displaymath}
(Z^U\times X^W_{1,2})\cup(X^U_{1,2}\times Z^W)
\end{displaymath}
is an edge of $G$.
\end{claim}

\begin{proof}
To prove the first assertion, suppose that $y^U_2y^W_2$ is an edge of
$G$, where $y^U_2\in Y^U_2$ and $y^W_2\in Y^W_2$.  By the definition
of $Y^U_2$, there is an $x^W\in X^W_{1,2}$ such that
$y^U_2x^W\in\aE_G$.  Claim~\ref{claim2} rules out $x^W\in X^W_1$, so
$x^W\in X^W_2$; write $x^W=x^W_2$.  Similarly, there is an
$x^U_2\in X^U_2$ such that $x^U_2y^W_2\in\aE_G$.  Choose
$y^U_1,y^W_1,x^W_1,x^U_1$ as in the proof of Claim~\ref{claim2}.  Then
\begin{displaymath}
u_p-x^W_1\cdots y^U_1-x^W_2\cdots y^U_2-y^W_2\cdots x^U_2-y^W_1\cdots x^U_1-w_q
\end{displaymath}
is an edge-abundant alternating path of length $9$, a contradiction.
For the second assertion, if $z^U\in Z^U$ had a non-edge to
some vertex of $X^W_{1,2}$, then $z^U$ would belong to
$Y^U_2=S^U\cap\Nec_G(X^W_{1,2})$, contradicting
$Z^U=S^U\setminus Y^U_2$.  Thus every chord in
$Z^U\times X^W_{1,2}$ is an edge of $G$.  The proof for
$X^U_{1,2}\times Z^W$ is symmetric.
\end{proof}

We next verify that the sets defined so far are disjoint as required.
The intersections $X^U_1\cap X^U_2$ and
$X^W_1\cap X^W_2$ are empty by definition.  If
$x\in X^U_1\cap Y^U_1$, then $x$ has a non-edge to some
vertex of $X^W_1$, contrary to Claim~\ref{claim2}; the same argument gives
$X^W_1\cap Y^W_1=\emptyset$.  If $x\in Y^U_1\cap X^U_2$, then $x$ has
a neighbor in $Y^W_1$, contrary to Claim~\ref{claim1}; again the $W$-side is
symmetric.  Moreover, $u_p$ belongs to none of
$X^U_1,Y^U_1,X^U_2$.  Membership in $X^U_1$ would mean that $u_pw_q$
is an edge.  Membership in $Y^U_1$ is impossible because every pair
from $u_p$ to $X^W_1=\Gamma_G(u_p)$ is an edge, while membership in
$X^U_2$ would force a vertex in $X^W_1\cap Y^W_1$.  The argument for $w_q$ is
symmetric.  Thus $X^U_1,Y^U_1,X^U_2$ are pairwise disjoint, and so are
$X^W_1,Y^W_1,X^W_2$.  Together with the definitions of $S^U$ and
$S^W$, this gives the partitions
\begin{displaymath}
U=\{u_p\}\uplus X^U_1\uplus Y^U_1\uplus X^U_2\uplus S^U,\qquad
W=\{w_q\}\uplus X^W_1\uplus Y^W_1\uplus X^W_2\uplus S^W.
\end{displaymath}

We now begin to modify $G$.  The sequence $\bd(G_L)-(e_p,e_q)$ of the final
graph will lie in the prescribed simple region.  The intermediate graphs,
however, need not satisfy the hypothesis excluding short edge-abundant
alternating paths from $u_p$ to $w_q$.

We adapt the ``downward twist'' construction from the proof of
Theorem~2.3 in~\cite{Simple-fully}.  The sets defined above remain
fixed throughout the construction.  Put $G_0=G$.  Suppose that
$G_\ell$ has been constructed.  If there exist vertices satisfying
\begin{displaymath}
a_\ell\in Z^U,\quad b_\ell\in Y^W_2,\quad c_\ell\in Z^W,\qquad
a_\ell b_\ell\in E(G_\ell),\quad a_\ell c_\ell\in \aE_{G_\ell},
\end{displaymath}
then perform the hinge-flip
\begin{displaymath}
a_\ell b_\ell\Rightarrow a_\ell c_\ell
\end{displaymath}
and set
\begin{displaymath}
G_{\ell+1}=G_\ell-a_\ell b_\ell+a_\ell c_\ell.
\end{displaymath}
We also allow the symmetric move obtained by interchanging $U$ and $W$.
Either operation is called an admissible downward twist.  Repeat this
step for as long as an admissible twist exists.

\begin{figure}[H]
\begin{center}
\begin{tikzpicture}[scale=.85]
\tikzstyle{vertex}=[draw,circle,fill=black,minimum size=3,inner sep=0]
\def\beforeZULeft{-.75}
\def\beforeZURight{1.45}
\def\beforeZUCenter{.35}
\def\beforeZUVertexX{1.1}
\def\beforeZULabelX{.65}
\def\beforeSideLeft{2.55}
\def\beforeSideRight{4.75}
\def\beforeSideCenter{3.65}
\def\beforeSideVertexX{2.9}
\def\beforeSideLabelX{3.3}
\def\afterZULeft{5.75}
\def\afterZURight{7.95}
\def\afterZUCenter{6.85}
\def\afterZUVertexX{7.6}
\def\afterZULabelX{7.15}
\def\afterSideLeft{9.05}
\def\afterSideRight{11.25}
\def\afterSideCenter{10.15}
\def\afterSideVertexX{9.4}
\def\afterSideLabelX{9.8}
\def\zuBottom{1.25}
\def\zuTop{2.55}
\def\zuLabelY{2.32}
\def\zuVertexY{1.65}
\def\zuTextY{1.55}
\def\ywBottom{3}
\def\ywTop{4.3}
\def\ywLabelY{4.05}
\def\ywVertexY{3.35}
\def\ywTextY{3.25}
\def\zwBottom{-.05}
\def\zwTop{1.25}
\def\zwLabelY{1}
\def\zwVertexY{.35}
\def\zwTextY{.25}

\node at (1.8,4.55) {$G_\ell$};
\node at (8,4.55) {$G_{\ell+1}$};
\draw[loosely dotted,very thick,->] (4.85,2.1) -- (5.45,2.1);

\draw (\beforeZULeft,\zuBottom) rectangle (\beforeZURight,\zuTop);
\node at (\beforeZUCenter,\zuLabelY) {$Z^U$};
\node[vertex] (a0) at (\beforeZUVertexX,\zuVertexY) {};
\node at (\beforeZULabelX,\zuTextY) {$a_\ell$};
\draw (\beforeSideLeft,\ywBottom) rectangle (\beforeSideRight,\ywTop);
\node at (\beforeSideCenter,\ywLabelY) {$Y^W_2$};
\node[vertex] (b0) at (\beforeSideVertexX,\ywVertexY) {};
\node at (\beforeSideLabelX,\ywTextY) {$b_\ell$};
\draw (\beforeSideLeft,\zwBottom) rectangle (\beforeSideRight,\zwTop);
\node at (\beforeSideCenter,\zwLabelY) {$Z^W$};
\node[vertex] (c0) at (\beforeSideVertexX,\zwVertexY) {};
\node at (\beforeSideLabelX,\zwTextY) {$c_\ell$};
\draw (a0) -- (b0);
\draw[dashed] (a0) -- (c0);
\draw[loosely dotted,very thick,->] (\beforeSideCenter,2.85) -- (\beforeSideCenter,1.45);

\draw (\afterZULeft,\zuBottom) rectangle (\afterZURight,\zuTop);
\node at (\afterZUCenter,\zuLabelY) {$Z^U$};
\node[vertex] (a1) at (\afterZUVertexX,\zuVertexY) {};
\node at (\afterZULabelX,\zuTextY) {$a_\ell$};
\draw (\afterSideLeft,\ywBottom) rectangle (\afterSideRight,\ywTop);
\node at (\afterSideCenter,\ywLabelY) {$Y^W_2$};
\node[vertex] (b1) at (\afterSideVertexX,\ywVertexY) {};
\node at (\afterSideLabelX,\ywTextY) {$b_\ell$};
\draw (\afterSideLeft,\zwBottom) rectangle (\afterSideRight,\zwTop);
\node at (\afterSideCenter,\zwLabelY) {$Z^W$};
\node[vertex] (c1) at (\afterSideVertexX,\zwVertexY) {};
\node at (\afterSideLabelX,\zwTextY) {$c_\ell$};
\draw[dashed] (a1) -- (b1);
\draw (a1) -- (c1);
\end{tikzpicture}
\end{center}
\caption{One downward twist.  In $G_\ell$ the pair $a_\ell b_\ell$ is an edge and $a_\ell c_\ell$ is a non-edge; after the hinge-flip, the edge is moved from $Y^W_2$ to $Z^W$.  The symmetric operation is obtained by interchanging the two vertex classes.}
\label{fig:downward-twist}
\end{figure}
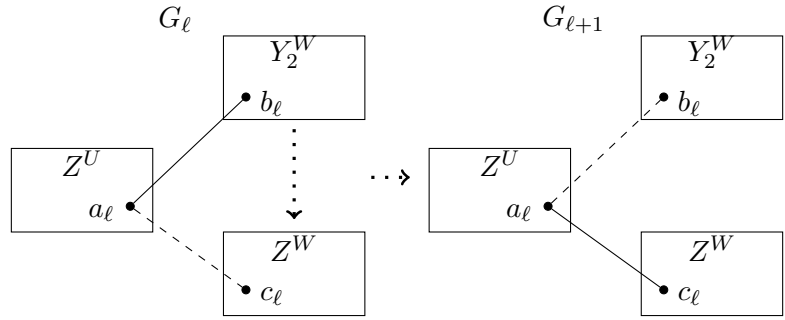

Each admissible downward twist removes an edge and adds a
non-edge, so the resulting graph is again simple and bipartite.  The added
pair is a chord; hence every $G_\ell$ remains $\F$-free.  The process
terminates because each twist adds one edge between $Z^U$ and $Z^W$,
while it removes an edge from either $Z^U\times Y^W_2$ or
$Y^U_2\times Z^W$; thus no edge of $Z^U\times Z^W$ is ever removed.
Let $G_L$ be the final graph and
write $\bd(G_L)=(C_L,D_L)$.
In the notation introduced above,
\begin{displaymath}
\bd(G_L)-(e_p,e_q)=(C_L-e_p,D_L-e_q).
\end{displaymath}

\begin{claim}\label{claim:twist-preservation}
For every $0\le \ell\le L$, all incidences asserted in
Claims~\ref{claim1}--\ref{claim3}, both edges and non-edges,
remain valid with $G_\ell$ in place of $G$.  Moreover,
\begin{displaymath}
\Gamma_{G_\ell}(u_p)=X^W_1,
\qquad
\Gamma_{G_\ell}(w_q)=X^U_1.
\end{displaymath}
\end{claim}
\begin{proof}
The sets in the original partition remain fixed throughout the
construction.  Relative to $G$, a downward twist can change only pairs in
\begin{displaymath}
(Z^U\times Y^W_2)\cup(Y^U_2\times Z^W)\cup(Z^U\times Z^W).
\end{displaymath}
It can only delete edges in the first two sets and add edges in the
third.  None of these pairs belongs to a class specified as edges or
non-edges in Claims~\ref{claim1}--\ref{claim3}.  Thus every incidence
from those claims is preserved.  Finally, no twist is incident with
$u_p$ or $w_q$, so their neighborhoods also remain unchanged.
\end{proof}

\begin{claim}\label{claim4}
\emph{The sequence $\bd(G_L)-(e_p,e_q)$ belongs to the prescribed simple region:}
\begin{equation}\label{eq:expanded-stays-in-region}
(C_L-e_p,D_L-e_q)\in\fD(\Sigma;n,\Gamma;m,\Delta;\F).
\end{equation}
\end{claim}
\begin{proof}
Each admissible downward twist preserves the number of edges, and no
twist is incident with $u_p$ or $w_q$.  Hence, after subtracting $e_p$
and $e_q$, the total sums are again $\Sigma$, and the $p$-th and
$q$-th coordinates are restored to their original values in $C$ and
$D$.  Moreover, only vertices in $S^U\cup S^W$ can have their degrees
changed.  Thus all degree bounds remain valid except possibly at
vertices in $S^U\cup S^W$.

Let $u\in S^U$.  By Claim~\ref{claim2}, every chord from $u$ to $X^W_1$
is an edge of $G$, and admissible downward twists never change such
edges.  Since $\F$ is a matching, at most one pair incident with $u$ is
forbidden.  Moreover,
\begin{displaymath}
\left |X^W_1 \right |=\deg_G(u_p)=c_p+1\ge \gamma_2+1.
\end{displaymath}
Consequently, at least $|X^W_1|-1$ pairs from $u$ to $X^W_1$ are
chords, and all of them are edges.  Hence
\begin{displaymath}
\deg_{G_L}(u)\ge |X^W_1|-1\ge\gamma_2.
\end{displaymath}

For the upper bound, first note that $u$ has no neighbors in
$\{w_q\}\cup Y^W_1$: it is not adjacent to $w_q$ by the definition of
$S^U$, and it is not adjacent to $Y^W_1$ by Claim~\ref{claim1}.  No admissible
downward twist creates an edge from $u$ to $\{w_q\}\cup Y^W_1$.  Hence
\begin{displaymath}
\deg_{G_L}(u)\le \left |X^W_{1,2}\right |+\left |S^W\right |.
\end{displaymath}
It remains to bound the right-hand side.  The set $X^U_1$ is nonempty,
since $\left |X^U_1\right |=\deg_G(w_q)=d_q+1$.  Choose $x^U_1\in X^U_1$.
The pair $u_pw_q$ is not an edge; otherwise it would itself be an
edge-abundant alternating path of length $1$ from $u_p$ to $w_q$.
Thus $x^U_1\ne u_p$, and so $\deg_G(x^U_1)\le \gamma_1$.  Put
$A=X^W_{1,2}\cup S^W$.  By Claim~\ref{claim2}, every chord from
$x^U_1$ to $A$ is an edge of $G$.  Since $\F$ is a matching, at most
one pair from $x^U_1$ to $A$ is forbidden, so $x^U_1$ has at least
$|A|-1$ neighbors in $A$.  Moreover, $w_q\notin A$ and $x^U_1w_q$ is
an edge by the definition of $X^U_1$.  Therefore
\begin{displaymath}
\deg_G(x^U_1)\ge (|A|-1)+1
=\left |X^W_{1,2}\right |+\left |S^W \right |.
\end{displaymath}
Consequently,
\begin{displaymath}
\deg_{G_L}(u)\le \left |X^W_{1,2}\right |+\left |S^W \right |\le \gamma_1.
\end{displaymath}
The argument for the $W$-side is symmetric; we include the details.
Consider $w\in S^W$.  Claim~\ref{claim2} and the matching
hypothesis give
\begin{displaymath}
\deg_{G_L}(w)\ge |X^U_1|-1=d_q\ge\delta_2.
\end{displaymath}
For the upper bound, $w$ has no neighbors in
$\{u_p\}\cup Y^U_1$, and no twist creates such a neighbor.  Choose
$x^W_1\in X^W_1$.  Then $x^W_1\ne w_q$,
$\deg_G(x^W_1)\le\delta_1$, every chord from $x^W_1$ to
$X^U_{1,2}\cup S^U$ is an edge, and $u_px^W_1$ is an additional edge.
The same count as above yields
\begin{displaymath}
\deg_{G_L}(w)\le |X^U_{1,2}|+|S^U|
\le\deg_G(x^W_1)\le\delta_1.
\end{displaymath}
Thus
\begin{displaymath}
\delta_2\le \deg_{G_L}(w)\le \delta_1
\end{displaymath}
for every $w\in S^W$.  This proves the claim.
\end{proof}

Finally, split $Z^U$ and $Z^W$ according to adjacency to the second
$Y$-block on the opposite side.  Let
\begin{displaymath}
X^U_3=\Gamma_{G_L}(Y^W_2)\cap Z^U,\qquad
X^W_3=\Gamma_{G_L}(Y^U_2)\cap Z^W,
\end{displaymath}
and
\begin{displaymath}
R^U=Z^U\setminus X^U_3,\qquad
R^W=Z^W\setminus X^W_3.
\end{displaymath}
Thus $Z^U=X^U_3\uplus R^U$ and $Z^W=X^W_3\uplus R^W$.

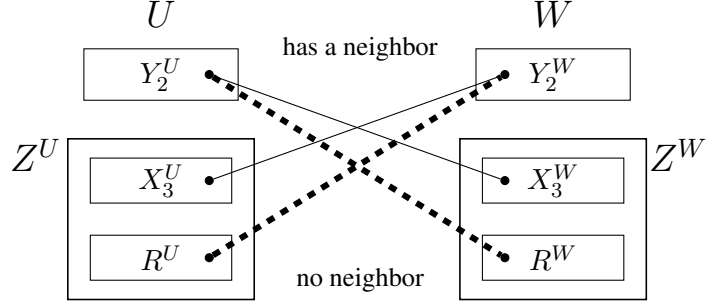
\begin{figure}[H]
\begin{center}
\begin{tikzpicture}[scale=.85]
\tikzstyle{vertex}=[draw,circle,fill=black,minimum size=3,inner sep=0]
\def\uYLeft{-.2}
\def\uYRight{2.2}
\def\uCenter{1}
\def\uVertexX{1.75}
\def\wYLeft{5.9}
\def\wYRight{8.3}
\def\wCenter{7.1}
\def\wVertexX{6.35}
\def\yTwoBottom{3.55}
\def\yTwoTop{4.35}
\def\yTwoCenter{3.95}
\def\uOuterLeft{-.45}
\def\uOuterRight{2.45}
\def\uOuterLabelX{-.95}
\def\wOuterLeft{5.65}
\def\wOuterRight{8.55}
\def\wOuterLabelX{9.05}
\def\outerBottom{.45}
\def\outerTop{2.95}
\def\outerLabelY{2.7}
\def\uInnerLeft{-.1}
\def\uInnerRight{2.1}
\def\wInnerLeft{6}
\def\wInnerRight{8.2}
\def\xThreeBottom{1.95}
\def\xThreeTop{2.65}
\def\xThreeCenter{2.3}
\def\rBottom{.75}
\def\rTop{1.45}
\def\rCenter{1.1}

\node at (\uCenter,4.9) {\Large $U$};
\node at (\wCenter,4.9) {\Large $W$};

\draw (\uYLeft,\yTwoBottom) rectangle (\uYRight,\yTwoTop);
\node at (\uCenter,\yTwoCenter) {$Y^U_2$};
\node[vertex] (uN) at (\uVertexX,\yTwoCenter) {};
\draw (\wYLeft,\yTwoBottom) rectangle (\wYRight,\yTwoTop);
\node at (\wCenter,\yTwoCenter) {$Y^W_2$};
\node[vertex] (wN) at (\wVertexX,\yTwoCenter) {};

\draw[line width=.6pt] (\uOuterLeft,\outerBottom) rectangle (\uOuterRight,\outerTop);
\node at (\uOuterLabelX,\outerLabelY) {\Large $Z^U$};
\draw (\uInnerLeft,\xThreeBottom) rectangle (\uInnerRight,\xThreeTop);
\node at (\uCenter,\xThreeCenter) {$X^U_3$};
\node[vertex] (u0) at (\uVertexX,\xThreeCenter) {};
\draw (\uInnerLeft,\rBottom) rectangle (\uInnerRight,\rTop);
\node at (\uCenter,\rCenter) {$R^U$};
\node[vertex] (u1) at (\uVertexX,\rCenter) {};

\draw[line width=.6pt] (\wOuterLeft,\outerBottom) rectangle (\wOuterRight,\outerTop);
\node at (\wOuterLabelX,\outerLabelY) {\Large $Z^W$};
\draw (\wInnerLeft,\xThreeBottom) rectangle (\wInnerRight,\xThreeTop);
\node at (\wCenter,\xThreeCenter) {$X^W_3$};
\node[vertex] (w0) at (\wVertexX,\xThreeCenter) {};
\draw (\wInnerLeft,\rBottom) rectangle (\wInnerRight,\rTop);
\node at (\wCenter,\rCenter) {$R^W$};
\node[vertex] (w1) at (\wVertexX,\rCenter) {};

\draw (u0) -- (wN);
\draw (uN) -- (w0);
\draw[line width=.8mm,dashed] (u1) -- (wN);
\draw[line width=.8mm,dashed] (uN) -- (w1);
\node at (4.1,\yTwoTop) {\small has a neighbor};
\node at (4.1,\rBottom) {\small no neighbor};
\end{tikzpicture}
\end{center}
\caption{Splitting $Z^U$ and $Z^W$ after the downward twists.  The set $X^U_3$ consists of the vertices in $Z^U$ with a neighbor in $Y^W_2$, while $R^U$ consists of the vertices in $Z^U$ with no such neighbor.  The $W$-side is symmetric.}
\label{fig:empty-zero-one}
\end{figure}

Since the downward-twist construction has terminated,
\begin{equation}\label{eq:expanded-max}
\left [Z^U\times X^W_3 \right ]\cap\mathcal C\subseteq E(G_L),
\qquad
\left [X^U_3\times Z^W \right ]\cap\mathcal C\subseteq E(G_L).
\end{equation}
Indeed, suppose that $u\in Z^U$, $w\in X^W_3$, and $uw$ is a chord.
By the definition of $X^W_3$, choose $r\in Y^U_2$ with
$rw\in E(G_L)$.  If $uw\notin E(G_L)$, then the symmetric hinge-flip
$wr\Rightarrow wu$ would still be admissible, contradicting the choice of
$G_L$.  Hence $uw\in E(G_L)$.  The other inclusion follows symmetrically.  By
definition,
\begin{equation}\label{eq:expanded-empty-tail}
\left [Y^U_2\times R^W \right ]\cap E(G_L)=\emptyset,\qquad
\left [R^U\times Y^W_2 \right ]\cap E(G_L)=\emptyset.
\end{equation}

Now set
\begin{displaymath}
X'_U=X^U_1\cup X^U_2\cup X^U_3,\quad
X'_W=X^W_1\cup X^W_2\cup X^W_3,
\end{displaymath}
\begin{displaymath}
Y'_U=Y^U_1\cup Y^U_2,\quad
Y'_W=Y^W_1\cup Y^W_2,
\end{displaymath}
and
\begin{displaymath}
R'_U=R^U,\quad R'_W=R^W.
\end{displaymath}
Together with the earlier partitions of $U$ and $W$, the identities
$Z^U=X^U_3\uplus R^U$ and $Z^W=X^W_3\uplus R^W$ give
\begin{displaymath}
U=\{u_p\}\uplus X'_U\uplus Y'_U\uplus R'_U,\qquad
W=\{w_q\}\uplus X'_W\uplus Y'_W\uplus R'_W.
\end{displaymath}

\begin{figure}[H]
\begin{center}
\begin{tikzpicture}[scale=1.0]
\tikzstyle{vertex}=[draw,circle,fill=black,minimum size=3,inner sep=0]
\def\uOuterLeft{-.9}
\def\uOuterRight{1.85}
\def\uOuterLabelX{-1.35}
\def\uInnerLeft{-.55}
\def\uInnerRight{1.45}
\def\uInnerCenter{.45}
\def\uVertexX{1.15}
\def\wOuterLeft{5.5}
\def\wOuterRight{8.25}
\def\wOuterLabelX{8.7}
\def\wInnerLeft{5.9}
\def\wInnerRight{7.9}
\def\wInnerCenter{6.9}
\def\wVertexX{6.2}
\def\xOuterBottom{5.15}
\def\xOuterTop{6.65}
\def\xOuterLabelY{6.35}
\def\xMainBottom{5.95}
\def\xMainTop{6.45}
\def\xMainCenter{6.2}
\def\xThreeBottom{5.35}
\def\xThreeTop{5.85}
\def\xThreeCenter{5.6}
\def\yOuterBottom{3.15}
\def\yOuterTop{4.65}
\def\yOuterLabelY{4.35}
\def\yOneBottom{3.95}
\def\yOneTop{4.45}
\def\yOneCenter{4.2}
\def\yTwoBottom{3.35}
\def\yTwoTop{3.85}
\def\yTwoCenter{3.6}
\def\rOuterBottom{1.75}
\def\rOuterTop{2.55}
\def\rOuterCenter{2.15}
\def\rInnerBottom{1.95}
\def\rInnerTop{2.35}
\def\xEdgeLineY{6.15}
\def\xTailLineY{5.25}
\def\yDashLineY{3.9}

\node at (-2,7.55) {\Large $U$};
\node at (9.5,7.55) {\Large $W$};
\node[vertex] (up) at (.5,7.05) [label=west:{$u_p$}] {};
\node[vertex] (wq) at (7.2,7.05) [label=east:{$w_q$}] {};

\draw[line width=.6pt] (\uOuterLeft,\xOuterBottom) rectangle (\uOuterRight,\xOuterTop);
\node at (\uOuterLabelX,\xOuterLabelY) {\Large $X'_U$};
\draw (\uInnerLeft,\xMainBottom) rectangle (\uInnerRight,\xMainTop);
\node at (\uInnerCenter,\xMainCenter) {$X^U_{1,2}$};
\node[vertex] (ku) at (\uVertexX,\xMainCenter) {};
\draw (\uInnerLeft,\xThreeBottom) rectangle (\uInnerRight,\xThreeTop);
\node at (\uInnerCenter,\xThreeCenter) {$X^U_3$};
\node[vertex] (rue0) at (\uVertexX,\xThreeCenter) {};

\draw[line width=.6pt] (\uOuterLeft,\yOuterBottom) rectangle (\uOuterRight,\yOuterTop);
\node at (\uOuterLabelX,\yOuterLabelY) {\Large $Y'_U$};
\draw (\uInnerLeft,\yOneBottom) rectangle (\uInnerRight,\yOneTop);
\node at (\uInnerCenter,\yOneCenter) {$Y^U_1$};
\node[vertex] (yu) at (\uVertexX,\yOneCenter) {};
\draw (\uInnerLeft,\yTwoBottom) rectangle (\uInnerRight,\yTwoTop);
\node at (\uInnerCenter,\yTwoCenter) {$Y^U_2$};
\node[vertex] (run) at (\uVertexX,\yTwoCenter) {};

\draw[line width=.6pt] (\uOuterLeft,\rOuterBottom) rectangle (\uOuterRight,\rOuterTop);
\node at (\uOuterLabelX,\rOuterCenter) {\Large $R'_U$};
\draw (\uInnerLeft,\rInnerBottom) rectangle (\uInnerRight,\rInnerTop);
\node at (\uInnerCenter,\rOuterCenter) {$R^U$};
\node[vertex] (rue1) at (\uVertexX,\rOuterCenter) {};

\draw[line width=.6pt] (\wOuterLeft,\xOuterBottom) rectangle (\wOuterRight,\xOuterTop);
\node at (\wOuterLabelX,\xOuterLabelY) {\Large $X'_W$};
\draw (\wInnerLeft,\xMainBottom) rectangle (\wInnerRight,\xMainTop);
\node at (\wInnerCenter,\xMainCenter) {$X^W_{1,2}$};
\node[vertex] (kw) at (\wVertexX,\xMainCenter) {};
\draw (\wInnerLeft,\xThreeBottom) rectangle (\wInnerRight,\xThreeTop);
\node at (\wInnerCenter,\xThreeCenter) {$X^W_3$};
\node[vertex] (rwe0) at (\wVertexX,\xThreeCenter) {};

\draw[line width=.6pt] (\wOuterLeft,\yOuterBottom) rectangle (\wOuterRight,\yOuterTop);
\node at (\wOuterLabelX,\yOuterLabelY) {\Large $Y'_W$};
\draw (\wInnerLeft,\yOneBottom) rectangle (\wInnerRight,\yOneTop);
\node at (\wInnerCenter,\yOneCenter) {$Y^W_1$};
\node[vertex] (yw) at (\wVertexX,\yOneCenter) {};
\draw (\wInnerLeft,\yTwoBottom) rectangle (\wInnerRight,\yTwoTop);
\node at (\wInnerCenter,\yTwoCenter) {$Y^W_2$};
\node[vertex] (rwn) at (\wVertexX,\yTwoCenter) {};

\draw[line width=.6pt] (\wOuterLeft,\rOuterBottom) rectangle (\wOuterRight,\rOuterTop);
\node at (\wOuterLabelX,\rOuterCenter) {\Large $R'_W$};
\draw (\wInnerLeft,\rInnerBottom) rectangle (\wInnerRight,\rInnerTop);
\node at (\wInnerCenter,\rOuterCenter) {$R^W$};
\node[vertex] (rwe1) at (\wVertexX,\rOuterCenter) {};

\draw (up) -- (kw);
\draw (wq) -- (ku);

\draw[line width=.8mm] (\uOuterRight,\xEdgeLineY) -- (\wOuterLeft,\xEdgeLineY);
\draw[line width=.8mm] (\uOuterRight,\xTailLineY) -- (\wOuterLeft,\rInnerTop);
\draw[line width=.8mm] (\uOuterRight,\rInnerTop) -- (\wOuterLeft,\xTailLineY);

\draw[line width=.8mm,dashed] (\uOuterRight,\yDashLineY) -- (\wOuterLeft,\yDashLineY);
\draw[line width=.8mm,dashed] (\uOuterRight,\yTwoBottom) -- (\wOuterLeft,\rInnerBottom);
\draw[line width=.8mm,dashed] (\uOuterRight,\rInnerBottom) -- (\wOuterLeft,\yTwoBottom);
\end{tikzpicture}
\end{center}
\caption{The final partition of $G_L$ as a bipartite blocking configuration.
}
\label{fig:final-partition}
\end{figure}
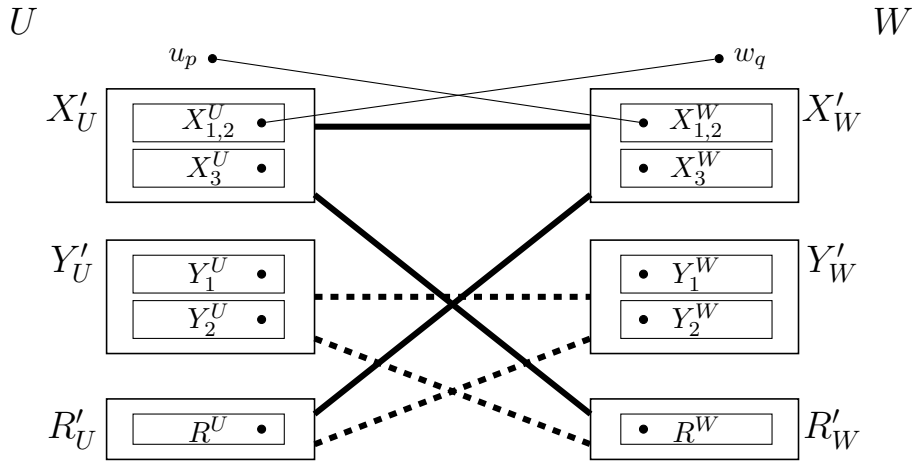

These partitions form a bipartite blocking configuration in $G_L$.
Indeed, by Claim~\ref{claim:twist-preservation}, all the relevant
incidences from Claims~\ref{claim1}--\ref{claim3} still hold in $G_L$.
Thus Claim~\ref{claim2}, the second part of Claim~\ref{claim3}, and
\eqref{eq:expanded-max} give conditions (a) and (b) of
Definition~\ref{def:bip-blocking}.  Similarly, Claim~\ref{claim1}, the
first part of Claim~\ref{claim3}, and
\eqref{eq:expanded-empty-tail} give conditions (c) and (d).  Finally,
no downward twist is incident with $u_p$ or $w_q$, and hence
\begin{displaymath}
\Gamma_{G_L}(u_p)=X^W_1\subseteq X'_W,\qquad
\Gamma_{G_L}(w_q)=X^U_1\subseteq X'_U.
\end{displaymath}

By Lemma~\ref{lem:bip-blocking-nonrealizable}, the sequence
$\bd(G_L)-(e_p,e_q)$, namely
\begin{displaymath}
(C_L-e_p,D_L-e_q)
\end{displaymath}
has no $\F$-free realization.  In view of~\eqref{eq:expanded-stays-in-region},
this sequence belongs to
\begin{displaymath}
\fD(\Sigma;n,\Gamma;m,\Delta;\F).
\end{displaymath}
Therefore the region is not fully $\F$-graphic, completing the proof of Theorem~\ref{tm:main-tree}.
\end{proof}

\begin{proof}[Proof of Theorem~\ref{tm:main}]
Fix $(C,D;\F)\in\fD$, and let
\begin{displaymath}
\Omega_{C,D,\F}=
\left\{(i,j,G'):\begin{array}{l}
1\le i\le n,\ 1\le j\le m,\\[1mm]
G'\in\G(C+e_i,D+e_j,\F)
\end{array}\right\}.
\end{displaymath}
Thus the indices are included as tags in $\Omega_{C,D,\F}$, and
\begin{equation}\label{eq:Omega-cardinality}
|\Omega_{C,D,\F}|=
\sum_{i=1}^n\sum_{j=1}^m \left|\G(C+e_i,D+e_j,\F)\right|.
\end{equation}

Fix a total order on all ordered paths of odd length at most $9$ in the
complete bipartite graph with classes $(U,W)$.  For
$(i,j,G')\in\Omega_{C,D,\F}$, Theorem~\ref{tm:main-tree} gives an
edge-abundant alternating path from $u_i$ to $w_j$ of length at most $9$.
Let $P(i,j,G')$ be the first such path in the fixed order.  Define
\begin{displaymath}
\Phi:\Omega_{C,D,\F}\longrightarrow \G(C,D,\F)
\end{displaymath}
by flipping every edge and non-edge along $P(i,j,G')$.  The internal vertices of this path retain their degrees, while $u_i$ and $w_j$ each lose one incident edge.  Since every added pair is a non-edge, $\Phi(i,j,G')$ is indeed an $\F$-free realization of $(C,D)$.

We now bound the fibers of $\Phi$.  Given $G^\dagger\in\G(C,D,\F)$ and the ordered path $P=P(i,j,G')$, one recovers $G'$ uniquely by flipping the consecutive pairs of $P$ in $G^\dagger$; moreover, the first and last
vertices of $P$ recover $i$ and $j$.  Hence the elements of $\Phi^{-1}(G^\dagger)$ have distinct selected paths.

Every possible selected path has the form
\begin{displaymath}
P=(u_{a_1},w_{b_1},\ldots,u_{a_k},w_{b_k})
\qquad (1\le k\le5).
\end{displaymath}
Pad each of the two ordered lists
$(u_{a_1},\ldots,u_{a_k})$ and $(w_{b_1},\ldots,w_{b_k})$ to length
five by repeating the final entry.  Because $P$ is a path, its entries
before the padding are distinct, so this gives an injective encoding of
all such ordered paths into $U^5\times W^5$.  Consequently,
\begin{displaymath}
|\Phi^{-1}(G^\dagger)|\le |U|^5|W|^5=(nm)^5.
\end{displaymath}
Therefore
\begin{displaymath}
|\Omega_{C,D,\F}|\le (nm)^5|\G(C,D,\F)|.
\end{displaymath}
Dividing by $|\G(C,D,\F)|$ and using \eqref{eq:Omega-cardinality} and~\eqref{eq:localP-F}, we obtain
\begin{displaymath}
\pppp{(C,D,\F)}\le (nm)^5.
\end{displaymath}
\end{proof}

\section{Fully graphic bipartite and directed regions - examples}\label{sec:examples}
\noindent
As we already mentioned, the fact that an infinite degree region is $P$-stable is very useful. However, proving directly such a statement is not an easy task. Typically it needs direct enumeration of graph realizations under complicated requirements, or tricky manipulations of graph realizations and "almost" realizations.

\medskip\noindent The notion of $P$-stability, in its current form, for  simple graphs was introduced by  Jerrum, McKay and Sinclair (\cite{JMS92}) in 1992, and they described the first $P$-stable regions in the same paper (Theorem 8.1 and Lemma 8.3). However, prior to 2026 there were only two known P-stable infinite \emph{bidegree} regions. Both were discovered (implicitly) in the paper \cite{linear} in 2018. (The authors even were not aware of the notion of $P$-stability.) These regions can be described as:
\begin{enumerate}[(PS1)]
\item Theorem 3:
\begin{equation}\label{eq:PB1}
\varphi_1(C,D)\ \equiv\ (\gamma_1-\gamma_2-1)\cdot (\delta_1- \delta_2 -1) \le\max\{ \gamma_2(n-\delta_1), \delta_2(m-\gamma_1)\}.
\end{equation}
\item Theorem 2:
\begin{equation}\label{eq:PB3}
\varphi_2(C,D)\ \equiv\ 2\le \gamma_1,\delta_1\le \sqrt{\Sigma/2}.
\end{equation}
\end{enumerate}
The first condition defines a very simple region since it applies for each feasible edge number, whereas the second region depends on $\Sigma$.  These conditions were
used directly to prove the rapid mixing behavior of the bipartite switch Markov chain. The recognition of the fact that these regions are $P$-stable provides a much easier proof via the general results of \cite{P-stable}.

There are also only two known P-stable directed regions, both were proved (implicitly) in the same paper (\cite{linear}). These are the following directed degree-sequence regions on the underlying set $X$:
\begin{enumerate}[(PS1)]
\setcounter{enumi}{2}
\item Theorem 5:
\begin{equation}\label{eq:PD1}
\varphi_3(C,D,\F) \equiv (\gamma_1- \gamma_2)\cdot (\delta_1-\delta_2) \le (2 - |X|) + \max
\left\{
  \begin{array}{ll}
   \gamma_2(|X|-\delta_1-1)+\delta_2+\gamma_1, & \\
    \delta_2(|X|-\gamma_1-1)+\gamma_2+\delta_1. &
  \end{array}
\right.
\end{equation}
\item Theorem 4:
\begin{equation}\label{eq:PD2}
\varphi_4(C,D,\F) \equiv 2\le \delta_1 \le \gamma_1\le \sqrt{\Sigma/2-4}.
\end{equation}
\end{enumerate}
Let's remark that \eqref{eq:PD1} is not in the form of a basic directed region  \eqref{eq:DSR-S} since, through the parameter $|X|$, it depends on the size of the partial pairing, since $|X|=n+m-|\F|$.

In this section we will extend all those results by increasing significantly the RHS  of those inequalities. We will also develop new type of $P$-stable regions, similar to \cite[Theorem 8.3]{JMS92}.

Preparing this article it was interesting to recognize, how similarly looked the bipartite and the directed results above. This similarity remains at these new results as well. The reason of these similarity lies in the use of the method of PPBGs.

In the next subsection we will present our fully graphic bidegree regions, while in the second subsection we will discuss some fully $\F$-free  graphic regions. It will be easy to recognize, that the results look alike, but the directed cases always require slightly stronger assumptions.

\subsection{Fully graphic bipartite regions}\label{sec:fully-bipart}
\noindent
At first we discuss the extension of theorem (PS1). Instead of maximum of two terms in the RHS, we take  their sum:

\begin{theorem}\label{th:b-fully1}
The very simple region $\fB(n,\Gamma; m, \Delta)$ is fully graphic if
\begin{equation}\label{eq:PB2'}
 (\gamma_1-\gamma_2-1)\cdot (\delta_1- \delta_2 -1) \le  \gamma_2(n-\delta_1) +\delta_2(m-\gamma_1)-3.
\end{equation}
\end{theorem}
\noindent
Indeed, one can note that the right hand side of \eqref{eq:PB2'} is essentially the sum of the two expressions on the right hand side of the inequality defining (PS1).  Therefore, as long as $\max\{\gamma_2(n-\delta_1), \delta_2(m-\gamma_1)\} \geq 3$, our region is an extension of that considered in (PS1), often to a significant degree.

\begin{lemma}\label{lem:alpha+1}
Using the notation of Definition~\ref{def:LEG}  if either
\begin{equation}\label{eq:strong2}
(\alpha+1)(\gamma_1 - \beta) \le  \delta_2(m-\beta)
\end{equation}
or
\begin{equation}\label{eq:strong2b}
(\beta+1)(\delta_1 - \alpha) \le  \gamma_2(n-\alpha)
\end{equation}
hold, then   $\fB(\Sigma; n,\Gamma; m, \Delta)$ is fully graphic.

Furthermore, in the special case of $\alpha = \delta_1-1$ (respectively $\beta = \gamma_1-1$) the slightly weaker inequality
\begin{equation}
    \alpha(\gamma_1-\beta) \leq \delta_2(m-\beta)
\end{equation}
respectively
\begin{equation}
 \beta(\delta_1-\alpha) \leq \gamma_2(n-\alpha)
\end{equation}
already implies that  $\fB(\Sigma; n,\Gamma; m, \Delta)$ is fully graphic.
\end{lemma}

\begin{proof}
Due to Theorem \ref{th:DI-LEG} we know that region $\fB(\Sigma; n,\Gamma; m, \Delta)$ is fully graphic when the $\lbb(\Sigma; n,\Gamma; m, \Delta)$  is graphic. By Theorem \ref{th:ZZ2} this is equivalent to showing the two inequalities
\begin{align}
\alpha \gamma_1 &\le \beta \min(\alpha,\delta_1) + \min \{\alpha,\delta\} + \delta_2 (m-\beta-1) \label{eq:alap1}\; \text{ and }\\
\alpha \gamma_1 + \gamma &\le \beta\min(\alpha+1,\delta_1) + \min \{\alpha+1,\delta\} + \delta_2 (m-\beta-1). \label{eq:alap2}
\end{align}
When $\alpha \geq \delta_1$ then the RHSs of both inequalities are just $\Sigma$, so both inequalities are automatically satisfied, hence the region is fully graphic as claimed.

When $\alpha< \delta_1$ then \eqref{eq:alap1} and \eqref{eq:alap2} take the following form:
\begin{align}
\alpha \gamma_1 &\le \alpha \beta + \min \{\alpha,\delta\} + \delta_2 (m-\beta-1) \label{eq:bounded1}\; \text{ and }\\
\alpha \gamma_1 + \gamma &\le (\alpha+1)\beta + \min \{\alpha+1,\delta\} + \delta_2 (m-\beta-1). \label{eq:bounded2}
\end{align}
Note that $\delta_2 \leq \min\{\alpha,\delta\} \leq \min\{\alpha+1,\delta\}$ hence it suffices to show
\begin{align}
\alpha (\gamma_1-\beta) &\le \delta_2 (m-\beta) \label{eq:bounded1b}\; \text{ and }\\
\alpha (\gamma_1-\beta) + \gamma &\le \beta +\delta_2 (m-\beta). \label{eq:bounded2b}
\end{align}
Since $\gamma \leq \gamma_1$, equation \eqref{eq:strong2} clearly implies both \eqref{eq:bounded1b} and \eqref{eq:bounded2b}.
Note that \eqref{eq:strong2b} also implies the fully graphic property due to the symmetry of the roles of the two parts.

Consider now the case $\alpha = \delta_1-1.$ Then the RHS of the inequality \eqref{eq:bounded2} still equals $\Sigma$ so it is automatically satisfied. Thus we only need to verify \eqref{eq:bounded1}, which, as we have seen, follows from \eqref{eq:bounded1b}.
The $\beta=\gamma_1-1$ case follows similarly.

\end{proof}

\begin{proof}[Proof of Theorem~\ref{th:b-fully1}] We prove by contradiction. Suppose that the very simple region $\fB(n,\Gamma;m,\Delta)$ is not fully graphic. Thus, for some $\max \{n\gamma_2, m\delta_2\} < \Sigma < \min\{n \gamma_1, m\delta_1\}$, the simple region $\fB(\Sigma;n,\Gamma;m,\Delta)$ is also not fully graphic.

Then, by Lemma \ref{lem:alpha+1}, we get that the parameters of $\lbb(\Sigma;n,\Gamma;m,\Delta)$ do not satisfy \eqref{eq:strong2}, nor \eqref{eq:strong2b}. That is, we have

\begin{align}
(\alpha+1) (\gamma_1-\beta) & \geq \delta_2 (m-\beta)+1\; \text{ and }\\
(\beta+1)(\delta_1 - \alpha) & \geq \gamma_2(n-\alpha)+1.
\end{align}
We also know from Lemma~\ref{lem:alpha+1} that
\begin{equation}\label{eq:alphabetabounds}
    \delta_2 < \alpha < \delta_1 \text{ and } \gamma_2 < \beta < \gamma_1
\end{equation}
hence Figure~\ref{fig1} depicts the correct ordering of these parameters.
\begin{figure}[ht]
        \centering
\begin{tikzpicture}[scale=.6]
    \filldraw[line width=0mm,fill=gray!30!white] (0,7) rectangle (2,8);
    \filldraw[line width=0mm,pattern=north east lines] (2,5) rectangle (4,7);
    \filldraw[line width=0mm,fill=gray!30!white] (4,3) rectangle (7,5);
    \filldraw[line width=0mm,pattern=north east lines] (7,0) rectangle (10,3);

    \draw [very thick] (0,0) -- (7,0);
    \draw [color=gray!50!black,very thick,decorate, decoration={bumps}] (-0.03,-.15) -- (10,-.05);
    \draw [color=gray!50!black,very thick,decorate, decoration={bumps}] (-0.03,-.15) -- (-.03,7);
    \node at (2,-.5) {$\delta_2$}; \node at (4,-.5) {$\alpha$}; \node at (7,-.5) {$\delta_1$};
    \node at (10,-.5) {$n$};
    \draw [very thick] (0,0) -- (0,8);
    \node at (-.5,3) {$\gamma_2$}; \node at (-.5,5) {$\beta$}; \node at (-.5,7) {$\gamma_1$};
    \node at (-.5,8) {$m$};

    \draw [very thick] (2,5) -- (2,8); \draw [very thick] (0,8) -- (2,8);
    \draw[very thin,dashed] (2,0) -- (2,5);
    \draw [color=gray!50!black,very thick,decorate, decoration={bumps}] (4.05,3) -- (4.05,7); \draw [color=gray!50!black,very thick,decorate, decoration={bumps}] (0,6.95) -- (4,6.95);
    \draw[very thin,dashed] (4,0)--(4,3);
    \draw [very thick] (7,-.1) -- (7,5); \draw [very thick] (2,5) -- (7,5);
    \draw[very thin,dashed] (0,5)--(2,5);
    \draw [color=gray!50!black,very thick,decorate, decoration={bumps}] (10.05,0) -- (10.05,3); \draw [color=gray!50!black,very thick,decorate, decoration={bumps}] (4,2.95) -- (10,2.95);
    \draw[very thin,dashed] (0,3)--(4,3);

    \node at (1,7.5) {$\mathbf A$}; \node at (3,6) {$\mathbf B$}; \node at (5.5,4) {$\mathbf C$};
    \node at (8.5,1.5) {$\mathbf D$}; \node at (1,6) {$\mathbf T_1$};\node at (5.5,1.5) {$\mathbf T_2$};
\end{tikzpicture}
\begin{tikzpicture}[scale=.6]
    \filldraw[line width=0mm,fill=gray!30!white] (0,7) rectangle (2,8);
    \filldraw[line width=0mm,pattern=north east lines] (2,5) rectangle (4,7);
    \filldraw[line width=0mm,fill=gray!30!white] (4,3) rectangle (7,5);
    \filldraw[line width=0mm,pattern=north east lines] (7,0) rectangle (10,3);

    \draw [color=gray!50!black,very thick,decorate, decoration={bumps}] (-0.03,-.15) -- (10,-.05);
    \draw [color=gray!50!black,very thick,decorate, decoration={bumps}] (-0.03,-.15) -- (-.03,7);
    \draw [very thick] (0,0) -- (7,0);
    \node at (2,-.5) {$\delta_2$}; \node at (4,-.5) {$\alpha$}; \node at (7,-.5) {$\delta_1$};
    \node at (10,-.5) {$n$};
    \draw [very thick] (0,0) -- (0,8);
    \node at (-.5,3) {$\gamma_2$}; \node at (-.5,5) {$\beta$}; \node at (-.5,7) {$\gamma_1$};
    \node at (-.5,8) {$m$};

    \draw [very thick] (2,5) -- (2,8); \draw [very thick] (0,8) -- (2,8);
    \draw[very thin,dashed] (2,0) -- (2,5);
    \draw [color=gray!50!black,very thick,decorate, decoration={bumps}] (4.05,3) -- (4.05,7);
    \draw [color=gray!50!black,very thick,decorate, decoration={bumps}] (0,6.95) -- (4,6.95);
    \draw[very thin,dashed] (4,0)--(4,3);
    \draw [very thick] (7,-.1) -- (7,5); \draw [very thick] (2,5) -- (7,5);
    \draw[very thin,dashed] (0,5)--(2,5);
    \draw [color=gray!50!black,very thick,decorate, decoration={bumps}] (10.05,0) -- (10.05,3); \draw [color=gray!50!black,very thick,decorate, decoration={bumps}] (4,2.95) -- (10,2.95);
    \draw[very thin,dashed] (0,3)--(4,3);

    \node at (1,7.5) {$\mathbf A$};\node at (3,6) {$\mathbf B$}; \node at (5.5,4) {$\mathbf C$};
    \node at (8.5,1.5) {$\mathbf D$};
    \node at (3,4) {$\mathbf T_3$};\node at (5.5,6) {$\mathbf T_4$};
\end{tikzpicture}
\caption{Let $[\mathbf A]$ denote the area of $\mathbf A$. Then $\alpha(\gamma_1 - \beta)=[\mathbf B]+[\mathbf T_1]$, and $\delta_2(m-\beta)=[\mathbf A]+[\mathbf T_1]$. Furthermore $\beta(\delta_1-\alpha)=[\mathbf C]+[\mathbf T_2]$ and $\gamma_2(n-\alpha)=[\mathbf D]+[\mathbf T_2]$. Lastly, $\Sigma - (\gamma-\gamma_2)$ is equal to the L-shaped area surrounded by bumpy lines. \label{fig1}}
\end{figure}
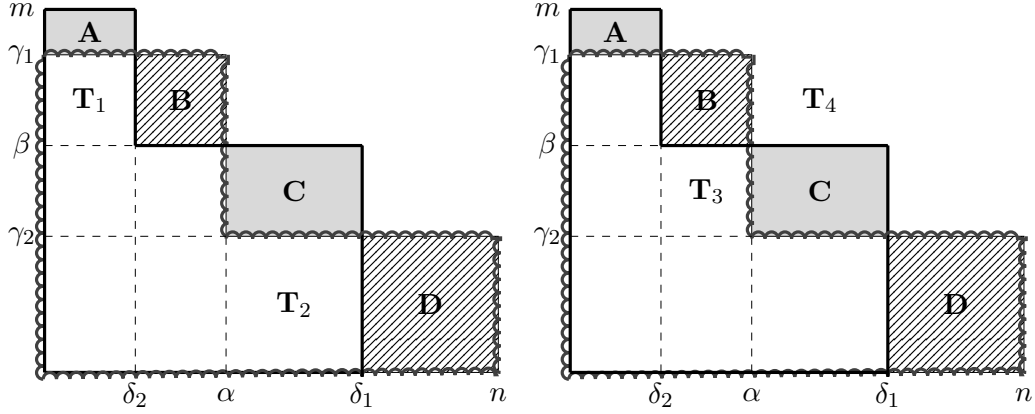

Using the notation from Figure \ref{fig1} (see the caption for explanation), the last two inequalities can be reformulated as:
\begin{align}\label{eq:must1}
\begin{split}
(\alpha+1)(\gamma_1 - \beta) \geq  \delta_2 (m-\beta)+1 \quad &\Leftrightarrow \quad [\mathbf B]+\gamma_1-\beta \geq [\mathbf A]+1, \\
(\beta+1)(\delta_1-\alpha) \geq  \gamma_2(n-\alpha)+1  \quad &\Leftrightarrow \quad [\mathbf C]+\delta_1-\alpha \geq [\mathbf D]+1.
\end{split}
\end{align}
To simplify the expressions, let us denote
\[
a:=\alpha-\delta_2, b:=\delta_1-\alpha \quad\text{and}\quad c:=\beta-\gamma_2, d:=\gamma_1-\beta.
\]
Then the inequalities in \eqref{eq:must1} can be written as
\begin{align}\label{eq:must2}
\begin{split}
a\cdot d + d \geq [\mathbf A]+1 \\
b \cdot c + b \geq [\mathbf D] +1,
\end{split}
\end{align}
whereas the assumption of the theorem, \eqref{eq:PB2'}, takes the form
\begin{equation}\label{eq:PBabcd}
    (a+b-1)(c+d-1) \leq [\mathbf A]+ [\mathbf D] - 3.
\end{equation}
Combining the sum in \eqref{eq:must2} with \eqref{eq:PBabcd} we arrive at
\begin{equation}
    ad + bc + d + b \geq (a+b-1)(c+d-1)+4 = ac + bd+ ad+ bc - a-b-c-d + 6
\end{equation}
or, equivalently,
\begin{equation}
    (a-1)(c-1) + (b-2)(d-2)+1 = ac - a - c  + bd - 2b - 2d + 6 \leq 0
\end{equation}
We already know (see \eqref{eq:alphabetabounds}) that $a,b,c,d \geq 1$. If, in fact, $b,d \geq 2$ we clearly get a contradiction and the proof is finished.

Thus the only case that is left to consider is when $b=1$ or $d=1$, that is $\alpha = \delta_1-1$ or $\beta=\gamma_1-1$. However, in these cases we know slightly stronger conditions from Lemma~\ref{lem:alpha+1}. Suppose $b=1$ but $d\geq 2$. Then
\begin{equation}
    \alpha(\gamma_1-\beta) \geq \delta_2(m-\beta)+1  \quad \Leftrightarrow \quad ad = [\mathbf{B}] \geq [\mathbf{A}]+1.
\end{equation}
Using this in the preceding argument in place of the first inequality in \eqref{eq:must2} we arrive at
\begin{equation}
        (a-1)(c-1) + (b-1)(d-2)+3 = ac - a - c  + bd - 2b - d + 6 \leq 0
\end{equation}
which is still a contradiction as now $d \geq 2.$ The cases when $d=1, b\geq 2$ and $d=b=1$ can be handled in a similar fashion, using the respective stronger conditions from Lemma~\ref{lem:alpha+1}.
\end{proof}

Next we extend (PS2) considerably: instead of proving that $\varphi_2$ in \eqref{eq:PB3} provides a fully graphic simple region, we show that the following larger region is also fully graphic.
\begin{theorem}\label{th:fully2}
The simple region $\fB[\varphi_{2'}]$, where
\begin{equation}\label{eq:biproot2}
\varphi_{2'} \quad \equiv \quad  (\gamma_1+1) \cdot (\delta_1+1) \le \Sigma
\end{equation}
is fully graphic.
\end{theorem}
\begin{proof} It is enough again to prove that the LEG is graphic. If $\alpha \geq \delta_1$ this holds automatically by Lemma~\ref{lem:alpha+1}. Otherwise, using the notation of Figure~\ref{fig1} we can see that
\begin{multline}
    0 \leq \Sigma - (\delta_1+1)(\gamma_1+1) = [\mathbf D] - [\mathbf C] - [\mathbf T_4] + (\gamma-\gamma_2) -\gamma_1- \delta_1+1 \\ \leq [\mathbf D] - [\mathbf C] - [\mathbf T_4] - (\delta_1-\alpha).
\end{multline}
Hence $[\mathbf D] \geq [\mathbf C] + (\delta_1-\alpha)$ which is equivalent to \eqref{eq:strong2b}, so by Lemma~\ref{lem:alpha+1} the region is fully graphic.
\end{proof}

A slight drawback of Lemma~\ref{lem:alpha+1} is that the condition involves $\alpha$ and $\beta$ which are implicitly determined by the other parameters. In what follows we present a sufficient condition for a simple region to be fully graphic that is based purely on its defining parameters. It is analogous to Jerrum, McKay and Sinclair second P-stability theorem in \cite{JMS92}, which also can be found as \cite[(P2)]{fully}.
\begin{theorem}\label{th:fully5}
The simple region $\fB(\Sigma;n, \Gamma; m, \Delta)$ is fully graphic if
\begin{multline}\label{eq:simplebipart}
\left(\Sigma - m\delta_2 - (\gamma_2-1)(\delta_1-\delta_2)\right)\left(n \gamma_1 - \Sigma - (n-\delta_1-1)(\gamma_1-\gamma_2)\right) \\ \leq \gamma_2 (n-\delta_1)(\delta_1-\delta_2) (\gamma_1-\gamma_2)
\end{multline}
\end{theorem}
\begin{proof}
If $\beta < \gamma_2$ or $\alpha \geq \delta_1$ then Lemma~\ref{lem:alpha+1} already implies that the region is fully graphic. Otherwise, from \eqref{eq:alfa} and \eqref{eq:alfa2} we have
\begin{multline*}
    \Sigma - m \delta_2 - (\gamma_2-1)(\delta_1-\delta_2) = (\delta_1-\delta_2)(\beta) + (\delta-\delta_2) - (\gamma_2-1)(\delta_1-\delta_2) \\ = (\delta_1-\delta_2)(\beta+1-\gamma_2) +(\delta-\delta_2) \geq (\beta-\gamma_2+1)(\delta_1-\delta_2)
\end{multline*}
    and
\begin{multline*}
    n\gamma_1 - \Sigma  - (n-\delta_1-1)(\gamma_1-\gamma_2) = (n-\alpha-1)(\gamma_1-\gamma_2) + (\gamma_1-\gamma) -(n-\delta_1-1)(\gamma_1-\gamma_2) \\= (\delta_1-\alpha)(\gamma_1-\gamma_2) + (\gamma_1-\gamma) \geq (\delta_1-\alpha)(\gamma_1-\gamma_2)
\end{multline*}
Combining the product of these two inequalities with the assumption of the theorem yields
\[ (\delta_1-\alpha)(\beta-\gamma_2+1)(\delta_1-\delta_2)(\gamma_1-\gamma_2) \leq \gamma_2(n-\delta_1)(\delta_1-\delta_2)(\gamma_1-\gamma_2),\] hence
\[ (\delta_1-\alpha)(\beta-\gamma_2+1) \leq \gamma_2(n-\delta_1).\]
Finally, adding $(\delta_1-\alpha)\gamma_2$ to both sides yields
\begin{align*}
(\beta+1)(\delta_1-\alpha) &\le  (n-\alpha) \gamma_2.
\end{align*}
Hence, by Lemma~\ref{lem:alpha+1}, the simple region $\fB(\Sigma;n, \Gamma; m, \Delta)$ is fully graphic.
\end{proof}

Viewing \eqref{eq:simplebipart} as a quadratic inequality in $\Sigma$, when its discriminant is negative we can conclude that the region $\fB(\Sigma;n, \Gamma; m, \Delta)$ is fully graphic for all values of $\Sigma$, hence the very simple region $\fB(n, \Gamma; m, \Delta)$ is fully graphic. This leads to a strengthening of Theorem~\ref{th:b-fully1}:

\begin{theorem}\label{cor:verysimple_fg}
The very simple region $\fB(n, \Gamma; m, \Delta)$ is fully graphic when $(m-\gamma_1)\delta_2\ge (\gamma_1-\gamma_2) + (\delta_1-\delta_2)$ and
\[ \left(\sqrt{[\mathbf A] - x-y} - \sqrt{[\mathbf D] }\right)^2 \leq xy \leq \left(\sqrt{[\mathbf A] -x-y} + \sqrt{[\mathbf D]}\right)^2,  \]
where $x= \delta_1-\delta_2$ and $y=\gamma_1-\gamma_2$ and $[\mathbf A]$ and $[\mathbf D]$ refer the respective areas in Figure~\ref{fig1}. Note that $xy = [\mathbf B]+[\mathbf C]+[\mathbf T_3]+[\mathbf T_4]$ in terms of areas in the figure.
\end{theorem}

\begin{proof}

The quadratic inequality \eqref{eq:simplebipart} takes the form $(\Sigma-A)( B-\Sigma)\le C$ where \[A = m\delta_2 + (\gamma_2-1)(\delta_1-\delta_2), \qquad B = n\gamma_1 - (n-\delta_1-1)(\gamma_1-\gamma_2), \] and
\[C = \gamma_2(n-\delta_1)(\delta_1-\delta_2)(\gamma_1-\gamma_2).\]
Hence the discriminant is $(A-B)^2-4C$. To understand when this is non-positive, it is helpful to think of it as a quadratic expression in $P = (\delta_1-\delta_2)(\gamma_1-\gamma_2)$. A simple calculation shows that $A-B = K - P$ where $K= (m-\gamma_1)\delta_2 - (n-\delta_1)\gamma_2 - (\delta_1-\delta_2) - (\gamma_1-\gamma_2)$, and $C = (n-\delta_1)\gamma_2 P$. Thus the discriminant is non-positive if and only if $(K-P)^2 \leq 4(n-\delta_1)\gamma_2 P$, or equivalently, when
\[ P^2 - 2(K+2(n-\delta_1)\gamma_2)P + K^2 \leq 0.\] The roots of this equation are
\[ P = \left(\sqrt{K+(n-\delta_1)\gamma_2} \pm \sqrt{(n-\delta_1)\gamma_2 }\right)^2 \] hence the original discriminant is non-positive when
\[ \left(\sqrt{K+(n-\delta_1)\gamma_2} - \sqrt{(n-\delta_1)\gamma_2 }\right)^2 \leq P \leq \left(\sqrt{K+(n-\delta_1)\gamma_2} + \sqrt{(n-\delta_1)\gamma_2 }\right)^2. \]
Finally, substituting back $K +(n-\delta_1)\gamma_2 = (m-\gamma_1)\delta_2 - (\delta_1-\delta_2) - (\gamma_1-\gamma_2)$ and $P= (\delta_1-\delta_2)(\gamma_1-\gamma_2)$ we get the claim.
\end{proof}

\subsection{Fully graphic directed regions}\label{sec:fully-direct}
\noindent
Let's recall the very simple region $\varphi_3[C,D,\F]$  from formula \eqref{eq:PD1}:
\[\varphi_3[(C,D,\F)] \equiv (\gamma_1- \gamma_2)\cdot (\delta_1-\delta_2) \le (2 - |X|) + \max
\left\{
  \begin{array}{ll}
   \gamma_2(|X|-\delta_1-1)+\delta_2+\gamma_1, & \\
    \delta_2(|X|-\gamma_1-1)+\gamma_2+\delta_1. &
  \end{array}
\right.
\]
This inequality does not take into consideration that the actual class sizes  in the $\F$-free bipartite representation are $n$ and $m$, and both $\le |X|.$ Fortunately inequality \eqref{eq:PD1} in (PS3) provides a direct proof, that the following, slightly transformed formula
\begin{equation}\label{eq:alakit}
\varphi_3^\prime[(C,D,\F)] \equiv (\gamma_1- \gamma_2)\cdot (\delta_1-\delta_2) \le 2 + \max
\left\{
  \begin{array}{ll}
   \gamma_2(n-\delta_1-1)+\delta_2+\gamma_1 -m, & \\
    \delta_2(m-\gamma_1-1)+\gamma_2+\delta_1 -n. &
  \end{array}
\right.
\end{equation}
also determines a P-stable infinite $\F$-free bidegree set. For this purpose it is enough to recognize that since $|X|=n+m-|\F|$ therefore the RHS of \eqref{eq:alakit} is smaller than the RHS of \eqref{eq:PD1}. Each degree  sequence satisfying \eqref{eq:alakit} also satisfies \eqref{eq:PD1}. Of course, this transformed formula provides a slightly weaker result than (PS3). \qed

\begin{example} One would think, based on the previous examples of the simple and bipartite versions, that $\varphi_3[(C,D,\F)]$ is a fully graphic region and hence its $P$-stability follows from our general theorem. However, quite surprisingly, the very simple region defined by \eqref{eq:alakit} is not fully graphic. Let $\gamma_1=\delta_1 = n-1=m-1$, and set $\gamma_2=\delta_2 = n-2k-2$ for some $k$ to be determined later. Then it is easy to see, that the directed degree sequence $(C_1,D_1)$ where
\begin{equation*}	
        C_1=D_1= (\underbrace{n-1,\ldots,n-1}_{n-2k-1}, \underbrace{n-2k-2, \ldots, n-2k-2}_{2k+1})
\end{equation*}
is non-graphic,  whereas the sequence $(C_2,D_2)$ where
\begin{equation*}	
C_2=D_2=(\underbrace{n-1,\ldots,n-1}_{n-4k-1}, \underbrace{n-2,\ldots,n-2}_{2k},\underbrace{n-2k-1, \ldots, n-2k-1}_{2k},n-2k-2)
\end{equation*}
is graphic. Indeed, this second claim is easily verified using Theorem~\ref{th:berger} or by simply constructing a graph with the given (or even easier, the complement of the given) directed degree sequence. Clearly, $(C_1,D_1)$ and $(C_2,D_2)$ are in the same simple region, as they have the same total edge count and same max/min degrees.

However, when $k$ is small enough (specifically when $(2k+1)^2 \leq n-2k-1$) then this region satisfies \eqref{eq:alakit} hence it is $P$-stable, while containing both graphic and not graphic sequences.
\end{example}

\noindent We can prove the following result:
\begin{theorem}\label{th:ELP}
Let \begin{equation}\label{eq:PD3'}
\varphi_{3}''[C,D,\F] \equiv (\gamma_1- \gamma_2)\cdot (\delta_1-\delta_2) \le  \gamma_2(n-\delta_1-1) + \delta_2(m-\gamma_1-1) -3
\end{equation}
then the very simple directed region $\fD[\varphi_3'']$ is fully graphic.
\end{theorem}
\begin{remark}
It is easy to see that for a large set of parameter choices the region defined by $\varphi_3''$  contains the region defined by $\varphi_3'.$ This
relation is false for (casually speaking) very small $\min\{\delta_2, \gamma_2\}$ or very large $\max\{\delta_1,\gamma_1\}$ parameter pairs. We leave the details for the diligent readers.
\end{remark}

We begin by stating and proving an analogue of Lemma~\ref{lem:alpha+1} for the directed case.

\begin{lemma}\label{lem:DILEG}
Using the notation of Definition~\ref{def:LEG}  if either
\begin{equation}\label{eq:Dstrong2}
(\alpha+1)(\gamma_1 - \beta+1) \le  \delta_2(m-\beta)
\end{equation}
or
\begin{equation}\label{eq:Dstrong2b}
(\beta+1)(\delta_1 - \alpha+1)  \le  \gamma_2(n-\alpha)
\end{equation}
hold, then   $\fD(\Sigma; n,\Gamma, m, \Delta; f)$ is fully graphic.

Furthermore, in the special case of $\alpha = \delta_1$  (respectively $\beta = \gamma_1$) the slightly weaker inequality
\begin{equation}\label{eq:Dweakeralpha}
    \alpha(\gamma_1-\beta+1) \leq \delta_2(m-\beta)
\end{equation}
respectively
\begin{equation}
 \beta(\delta_1-\alpha+1) \leq \gamma_2(n-\alpha)
\end{equation}
already implies that  $\fD(\Sigma; n,\Gamma, m, \Delta; f)$ is fully graphic.

Finally, in the special case of $\alpha = \delta_2$ (respectively $\beta= \gamma_2$) the following inequality is sufficient to prove graphicality:
\begin{equation}\label{eq:Dweakeralpha2}
    \gamma \leq \delta_2(m-\gamma_1) + \max\{\beta-\alpha-1,0\}
\end{equation}
respectively
\begin{equation}
\delta \leq \gamma_2(n-\delta_1) + \max\{\alpha-\beta-1,0\}
\end{equation}
\end{lemma}

\begin{proof}
Due to Theorem \ref{th:DI-LEG} we know that region $\fD(\Sigma; n,\Gamma, m, \Delta; f)$ is fully graphic when the $\lbb(\Sigma; n,\Gamma, m, \Delta; f)$  is graphic. By Theorem \ref{th:berger} this is equivalent to showing the two inequalities
\begin{align}
\alpha \gamma_1 &\le \sum_{i=1}^\alpha \min\{d_i,\alpha-1\} + \sum_{i=\alpha+1}^m \min\{d_i,\alpha\}\; \quad \text{and}\\ \label{eq:secondvariant}
\alpha \gamma_1 + \gamma &\le \sum_{i=1}^{\alpha+1} \min\{d_i,\alpha\} + \sum_{i=\alpha+2}^m \min\{d_i,\alpha+1\}.
\end{align}
If $\alpha > \delta_1$ then both right hand sides are equal to $\Sigma$ hence both inequalities automatically hold so we are done. If $\alpha = \delta_1$ then the second inequality is still automatically true so we only need to prove the first. From now on we assume $\alpha \leq \delta_1$ (and from symmetry also that $\beta \leq \gamma_1$). Note that in this case
\begin{multline}\sum_{i=1}^\alpha \min\{d_i,\alpha-1\} + \sum_{i=\alpha+1}^m \min\{d_i,\alpha\} \geq \sum_1^n \min\{d_i,\alpha\} - \alpha \\ =\beta \cdot \alpha + \min \{\alpha,\delta\} + (m-\beta-1)\min\{\delta_2 ,\alpha\} - \alpha \\\geq  (\beta-1) \cdot \alpha + (m-\beta)\min\{\delta_2,\alpha\} \end{multline} and similarly
\begin{multline} \sum_{i=1}^{\alpha+1} \min\{d_i,\alpha\} + \sum_{i=\alpha+2}^m \min\{d_i,\alpha+1\}\geq \sum_1^n \min\{d_i,\alpha+1\} - \alpha \\ =\beta \cdot (\alpha+1) + \min \{\alpha+1,\delta\} + (m-\beta-1)\min\{\delta_2 ,\alpha+1\} - \alpha-1 \\ \geq  (\beta-1) \cdot (\alpha+1) + (m-\beta)\min\{\delta_2,\alpha+1\} \end{multline}
As $\gamma < \gamma_1$, it suffices to show that
\begin{align}
\alpha \gamma_1 &\le (\beta-1) \cdot \alpha + (m-\beta)\min\{\delta_2,\alpha\}\; \text{ and } \label{eq:Dalap1}\\
(\alpha+1) \gamma_1  &\le (\beta-1) \cdot (\alpha+1) + (m-\beta)\min\{\delta_2,\alpha+1\}. \label{eq:Dalap2}
\end{align}
If $\alpha < \delta_2$ then the right hand sides are $(m-1)\cdot \alpha$ and $(m-1)\cdot (\alpha+1)$ respectively, so both inequalities follow from $\gamma_1 \leq m-1$ and we are done. Hence we may assume that $\alpha \geq \delta_2$. Then \eqref{eq:Dalap1} and \eqref{eq:Dalap2} take the following equivalent form:

\begin{align}
\alpha (\gamma_1-\beta+1) &\le \delta_2 (m-\beta) \label{eq:Dbounded1b}\; \text{ and }\\
(\alpha+1) (\gamma_1-\beta+1)  &\le \delta_2 (m-\beta). \label{eq:Dbounded2b}
\end{align}
Since $\beta \leq \gamma_1$, equation \eqref{eq:Dbounded2b} clearly implies \eqref{eq:Dbounded1b} and thus also graphicality of the $\lbb$, as claimed.

Next, based on our earlier observations, in the case of $\alpha=\delta_1$, already \eqref{eq:Dbounded1b} implies graphicality of the $\lbb$.

Finally, when $\alpha=\delta_2$, the first inequality follows from $m \geq \gamma_1+1$. Instead of \eqref{eq:Dalap2} we go back to \eqref{eq:secondvariant} and note that the right hand side there, for $\alpha=\delta_2$ is at least $m \delta_2 + \max\{\beta-\alpha-1,0\}$ hence \eqref{eq:Dweakeralpha2} is sufficient to imply full graphicality.

The role-reversed cases follow from the roles being symmetric.
\end{proof}

\begin{proof}[Proof of Theorem~\ref{th:ELP}]
According to Lemma~\ref{lem:fdependence} it is enough to consider $|\F|=\min\{m,n\}$ and by Lemma~\ref{lem:canonical-matching} we can assume that the matching $\F = \F_0$ from that lemma.

We proceed by contradiction, similarly to the proof of Theorem~\ref{th:b-fully1}, and again using
the notation of Figure~\ref{fig1}.  As observed in
Remark~\ref{th:alfa}, the required inequalities hold automatically if
$\alpha>\delta_1$ or $\alpha<\delta_2$, and likewise after exchanging
the two classes.  We may therefore assume \eqref{eq:alfa1}.

Suppose the region in \eqref{eq:PD3'} is not fully graphic. Then there is a specific set of values $(m,\Gamma, m,\Delta)$ satisfying $\varphi_3''$ and a $\Sigma$ in the range $\max \{n\gamma_2, m\delta_2\} < \Sigma < \min\{n \gamma_1, m\delta_1\}$ such that $\fD(\Sigma; n, \Gamma, m, \Delta; f)$ is not fully graphic.

Thus, by Lemma~\ref{lem:DILEG}, neither \eqref{eq:Dstrong2} nor \eqref{eq:Dstrong2b} is satisfied, which means that
\begin{align*}
(\alpha+1)(\gamma_1 - \beta+1) &\geq \delta_2(m-\beta) +1 \; \text{ and}\\
(\beta+1)(\delta_1 - \alpha+1)  &\geq  \gamma_2(n-\alpha)+1.
\end{align*}
This can be also written as
\begin{align*}
(\alpha-\delta_2+1)(\gamma_1 - \beta+1) &\geq \delta_2(m-\gamma_1) -\delta_2 +1 \; \text{ and}\\
(\beta-\gamma_2+1)(\delta_1 - \alpha+1)  &\geq  \gamma_2(n-\gamma_1)-\gamma_2+1.
\end{align*}
Using the notation from before: \[a = \alpha-\delta_2, b= \delta_1-\alpha, c= \beta-\gamma_2, d= \gamma_1-\beta,\] we thus have so far that
\begin{align*}
    (a+1)(d+1) &\geq \dD - \delta_2 +1 \; \text{ and}\\
    (c+1)(b+1) &\geq \aA - \gamma_2 +1
\end{align*}
Finally, the inequality defining the region takes the form of
\[ (a+b)(c+d) \leq \aA+\dD - \gamma_2- \delta_2-3\]
This, combined with the previous two, yields
\[(a+b)(c+d)+3 \leq \aA + \dD -\delta_2-\gamma_2 \leq ad + a+d+bc + b+c, \] which simplifies to
\[ (a-1)(c-1)+(b-1)(d-1)+ 1 \leq 0 \]
If $a,b,c,d > 0$, this is clearly impossible, thus the LEG is indeed graphic and so the region is fully graphic as claimed.

When one of these gaps is 0, we need to use the appropriate stronger condition from Lemma~\ref{lem:DILEG}. For example, suppose $b=0$, that is $\alpha=\delta_1$. Then by the lemma \eqref{eq:Dweakeralpha} has to fail, so we get
\[ \alpha(\gamma_1-\beta+1) \geq \delta_2(m-\beta)+1,\] or equivalently
\[ (\alpha-\delta_2)(\gamma_1-\beta+1) \geq \delta_2(m-\gamma_1)-\delta_2 + 1,\] which can again be written as
\[ a (d+1)  \geq \dD - \delta_2 +1.\]
Adding $(c+1)(b+1) \geq \aA - \gamma_2 +1 $ and applying $(a+b)(c+d) \leq \aA +\dD - \gamma_2- \delta_2 -3$ then yields

\[ (a-1)(c-1)+b(d-1) = ac + bd - a- b-c +1 \]

The other cases follow in a similar fashion, though when $a$ or $c$ is 0 one also needs to use the identity
\[ \alpha \gamma_1 + \gamma + (n-\alpha-1)\gamma_2 = \Sigma = \beta \delta_1\ + \delta + (m-\beta-1)\delta_2. \]

\end{proof}

It is also possible to prove results analogous to Theorem~\ref{th:fully5} and ~\ref{cor:verysimple_fg}. However we think these do not give further real insight to fully graphic directed regions, therefore we leave these to the interested readers.

\paragraph{AI disclosure:} ChatGPT 6 was used to refine the exposition, both by generating reports and making minor edits to the manuscript. It was also used to aid specific calculations, where mentioned. All significant ideas and methods were obtained without any AI use.

\bibliographystyle{plain}

\end{document}